\documentclass[11pt,reqno]{amsart}

\makeatletter

\usepackage{amssymb}
\usepackage{latexsym}
\usepackage{amsbsy}
\usepackage{amsfonts}
\usepackage{dsfont,txfonts
}
\usepackage{mathrsfs}
\usepackage{amsmath}%
\usepackage{amssymb}%
\usepackage{fancyhdr}%
\usepackage{ifthen}%
\usepackage{calc}%
\usepackage{lastpage}%
\usepackage{array}%
\usepackage{anysize}%
\usepackage{cases}
\usepackage{amsthm}
\usepackage{geometry}
\usepackage{color}

\newtheorem{theorem}{Theorem}[section]
\newtheorem{lemma}[theorem]{Lemma}

\newtheorem{proposition}[theorem]{Proposition}
\theoremstyle{definition}
\newtheorem{remark}[theorem]{Remark}

\numberwithin{equation}{section}

\newdimen\AAdi%
\newbox\AAbo%
\def\AAk#1#2{\setbox\AAbo=\hbox{#2}\AAdi=\wd\AAbo\kern#1\AAdi{}}%

\def\eqref#1{(\ref{#1})}
\def\eqlabel#1{\def\@currentlabel{#1}}

\def\formula#1{\def\@tempa{#1}\let\@tempb\theequation\def\theequation{%
\hbox{#1}}\def\@currentlabel{(\theequation)}$$}
\def\endformula{\leqno\hbox{(\@tempa)}$$\@ignoretrue\let\theequation\@tempb}

\def\given{\hskip5\p@\relax\vrule\@width.4\p@\hskip5\p@\relax}

\newcommand{\open}[1]{%
\par\normalfont\topsep6\p@\@plus6\p@\trivlist\item[\hskip\labelsep\itshape#1%
\@addpunct{.}]\ignorespaces}

\DeclareRobustCommand{\close}[1]{%
  \ifmmode 
  \else \leavevmode\unskip\penalty9999 \hbox{}\nobreak\hfill
  \fi
  \quad\hbox{$#1$}}

\newlength{\toskip}

\def\r{\right}
\def\l{\left}

\def\dsum{\displaystyle\sum}

\def\dlim{\displaystyle\lim}

\makeatother

\begin{document}
\date{\today}

\title[Functional Inequalities  on Time-Varying  Manifolds]{Diffusion Semigroup on  Manifolds with Time-Dependent Metrics}

\author[L.-J.  Cheng]{\textbf{Li-Juan Cheng}}
\address{{\bf {Li-Juan} CHENG}\\
Department of Applied Mathematics, Zhejiang University of Technology, \\
 Hangzhou 310023,   The People's Republic of China.}
\email{chenglj@zjut.edu.cn}

\begin{abstract}
Let $L_t:=\Delta_t +Z_t $, $t\in [0,T_c)$ on a differential manifold equipped with a complete geometric flow $(g_t)_{t\in [0,T_c)}$, where $\Delta_t$ is the Laplacian operator induced by the metric $g_t$ and $(Z_t)_{t\in [0,T_c)}$ is a family of $C^{1,\infty}$-vector fields. In this article, we present a number of equivalent  inequalities for the lower bound curvature  condition, which include
  gradient inequalities, transportation-cost inequalities, Harnack inequalities and other   functional inequalities for the semigroup associated with diffusion processes generated by $L_t$. To this end,  we
establish derivative formulae for  the associated semigroup
and  construct  couplings processes for these diffusion
processes by parallel displacement and reflection.
\end{abstract}

\maketitle

\textit{ Key words : $L_t$-diffusion processes, geometric flow, curvature, coupling, transportation-cost inequality, Harnack inequality, gradient inequality}

\bigskip

\textit{ MSC 2010 : 60J60, 58J65, 53C44.}
\bigskip

\section{Introduction}
In this article, we want to clarify the  connection between the behavior
of  distributions of  diffusion processes, and the geometry of
their underlying manifold carrying a geometric flow of complete Riemannian metrics, more precisely,
 a $d$-dimensional  differential manifold $M$ equipped with  a family of complete
Riemannian metrics $(g_t)_{t\in [0,T_c)}$ for some $T_c\in (0,\infty]$, which is $C^1$ in $t$.
 Let $\nabla^t$ and $\Delta_t$ be the Levi-Civita connection and the  Laplace-Beltrami
 operator associated with the metric $g_t$, respectively.
For simplicity, we take the notation: for $X, Y\in TM$,
\begin{align*}
    &\mathcal{R}_t^Z(X,Y):={\rm Ric}_t(X,Y)-\l<\nabla^t_XZ_t, Y\r>_t-\frac{1}{2}\partial_tg_t(X,Y),
\end{align*}
 where ${\rm Ric}_t$ is the Ricci curvature tensor with respect to the metric $g_t$, $(Z_t)_{t\in [0,T_c)}$ is a $C^{1,\infty}$-family of vector fields, and $\l<\cdot,\cdot\r>_t:=g_t(\cdot,\cdot)$.
Consider the diffusion process $X_t$ generated by   $L_t:=\Delta_t+Z_t$ (called  $L_t$-diffusion process $X_t$), which is assumed to be non-explosive before $T_c$. Let $\{P_{s,t}\}_{0\leq s\leq t<T_c}$ be the semigroup associated with $X_t$.
 The main work of this article is to study
 behaviors of $P_{s,t}$  by
using a new curvature condition, i.e. the low bound of $\mathcal{R}_t^Z$. Compared  with the usual Bakry-Emery's
curvature condition, it contains an additional term from the time derivative of the metric.

When the metric is independent of $t$, many excellent scholars did deep
research on the development of stochastic analysis on manifolds. In
\cite{Bismut,EL,Thalmaier},  derivative formulae for the associated diffusion
semigroup, known as the Bismut-Elworthy-Li formula, was given by
constructing a damped gradient operator. This formula was later applied to
 gradient estimates, transportation-cost inequalities and other important functional inequalities,
see for instance  \cite{Bakry94,Bakry, BE,BBG, Sturm, WEq,Wbook2}. Besides, coupling methods play
 an important role on stochastic analysis. In \cite{LR},   T. Lindvall and L.C.G. Rogers
 introduced the coupling processes for multi-dimensional diffusion processes, this idea was then extended by W.S. Kendall \cite{Ken} and M. Cranston \cite{Cr} to Riemannian manifolds, and  further  well refined in \cite{Cr, Wbook1}.   It is worth mentioning that by using coupling methods, M.-F. Chen and F.-Y. Wang \cite{CW94, CW97b, CW97a} gave  auxiliary results for estimates of the first eigenvalue on Riemannian manifolds.  Moreover,   based on constructing suitable coupling processes, some  equivalent important functional inequalities, including dimension-free Harnack inequalities, transportation-cost inequalities and gradient inequalities,  were presented for the  lower bound curvature condition (see
 e.g. \cite{WEq,Wbook2}  and reference therein).  For the case of manifolds  with
 boundary, we refer the readers to  \cite{ W09, W10,Wbook2} for details.  All these work motivate us to extend  the derivative
formulae and coupling processes to the time-inhomogeneous case, which are then applied to deriving some important functional inequalities for the diffusion semigroup.

Before moving on, let us briefly recall some known results in the time-inhomogeneous Riemannian setting.
In 2008, M. Arnaudon, K. Coulibaly and A. Thalmaier  \cite{ACT} first
  constructed the $g_t$-Brownian motions (i.e., the diffusion generated by $\frac{1}{2}\Delta_t$),  and established the Bismut
formula under the Ricci  flow, which in particular implies  gradient estimates of the associated heat semigroup, see also \cite{C11} for details. Next, by constructing horizontal diffusion processes, K. Coulibaly  \cite{ACT1} investigated
optimal transportation inequalities on time-varying manifolds. Moreover, K. Kuwada and R. Philipowski \cite{Ku} studied  the non-explosion of $g_t$-Brownian
motions   under some super Ricci flow. We would like to indicate that K.  Kuwada \cite{Ku3} has developed  coupling methods to estimate
the gradient of the semigroup by constructing a sequence  of time-inhomogeneous geodesic random walks.
Very recently, R. Haslhofer and A. Naber characterize
For
more development on  stochastic analysis in this setting, see \cite{KR} for  reviewing   the monotonicity of the
$\mathcal{L}$-transportation cost from  a probabilistic viewpoint;
see \cite{Ch1} for  stochastic analysis on the path space
over time-inhomogeneous manifolds.

The paper is organized as follows.  In Section 2, we will give  some basic notations and notions
for time-varying differential manifolds,
  introduce the construction of   $L_t$-diffusion processes and then
present several auxiliary results for the related diffusion semigroup.

In Section 3, we will show that, by constructing suitable local martingales,  a local derivative formula can be obtained by using local geometry of the manifold.  Moreover,  a globe derivative formula  will be studied here as well.  These arguments essentially follow from
A. Thalmaier \cite{Thalmaier} with some necessary modifications to our inhomogeneous context, since e.g.  geometric quantities are time-dependent and the underlying process is time-inhomogeneous.  As applications, we will study  local gradient inequalities for the diffusion semigroup.   By applying the ideas from \cite{TW}, the
main issue left should be to find a suitable testing function  to give a subtle upper estimate.

Next,  we will  show a number of
equivalent functional inequalities of the semigroup for the lower bound of $\mathcal{R}^Z_t$.  Inspired from \cite{Bakry,LR}, our first issue is to give some asymptotic formulae to characterize  $\mathcal{R}^Z_t$. Based on these formulae and the globe derivative formula established in Theorem \ref{1t3},   a number of equivalent gradient inequalities for  the lower bound of $\mathcal{R}^Z_t$ will be given in Theorem \ref{3t1} below. Then  we will  further consider  dimension-free Harnack inequalities  by using the gradient inequality established in Theorem \ref{3t1}. We  point out that when the metric is independent of $t$, these equivalences are well-known (see e.g. \cite{Bakry94,Bakry,BE,Wbook2}). Here, we make some necessary modification as the generater operator and geometric quantities are time-dependent.

In Section 4,   we will construct  coupling processes for $L_t$-diffusion processes by   solving
SDEs on $M\times M$ with
 singular coefficients on the space-time cut-locus. Compared with \cite{Ku3}, it looks   straightforward. When the metric is independent of $t$, our construction is due to \cite{Wbook1}.  In our setting,  as
the SDEs we consider below are non-autonomous  and the radiant process is non-differentiable on the space-time cut-locus, we need apply some results from \cite{Ku, MT} and find a suitable approximation to our desired process.
As applications,  we  will consider the transportation-cost inequality on time-inhomogeneous spaces.
We would like to indicate that very recently  the author uses coupling methods to investigate transportation-cost inequalities on path space of $L_t$-diffusion processes \cite{Ch2} and
dimension-free Harnack inequalities on time-varying manifolds with boundary \cite{Ch3}.

We end this section by making some conventions on the notations. Let $\mathcal{B}_b(M)$ be the set of all measurable functions,  $C^p_0(M)$ the set of all $C^p$-smooth real functions with compact supports on $M$ and $C_{c}^p(M)$ the set of all $C^p$-smooth real functions with constant outside a compact set.
  For any two-tensor $\mathbf{T}_t$ and any function $f$,
we write $\mathbf{T}_t\geq f$, if $\mathbf{T}_t(X,X)\geq
f\l<X,X\r>_t$ holds for $X\in TM$.
For any functions  $f$ and $\varphi$, respectively, defined on $[0,T_c)\times M$  and  $[0,T_c)\times M\times M$, we simply write $f_t(x):=f(t,x)$ and $\varphi_t(x,y):=\varphi(t,x,y)$, $t\in [0,T_c), x,y\in M$. $\mathbb{E}^{(s,x)}$ and $\mathbb{P}^{(s,x)}$ denote, respectively,  the expectation and the probability taken for the underlying process starting from $x$ at time $s$. When $s=0$, we simply write $\mathbb{E}^x:=\mathbb{E}^{(0,x)}$ and $\mathbb{P}^x=\mathbb{P}^{(0,x)}$.
\medskip

\section{Preliminaries}
Let $\mathcal{F}(M)$  be the frame bundle over $M$ and
$\mathcal{O}_{t}(M)$
the orthonormal frame bundle over $M$ with respect to the metric $g_t$.  Define $\mathbf{p}:\mathcal{F}(M)\rightarrow M$  the projection
from $\mathcal{F}(M)$ onto $M$.
 For any $u\in \mathcal{O}_t(M)$,
let $H^{t}_{X}(u)$ be the $\nabla^{t} $-horizontal lift
of $X\in T_{{\bf p}u}M$ and $H_{i}^t(u)=H_{ue_i}^t(u), i=1,2,\cdots, d$, where $\{e_{i}\}_{i=1}^{d}$ is the canonical orthonormal basis of
$\mathbb{R}^d$.
 Set $\{V_{\alpha, \beta}(u)\}_{\alpha,\beta=1}^d:=Tl_u(\exp(E_{\alpha,\beta}))$, $u\in \mathcal{F}(M)$ be the canonical basis
of vertical  fields over $\mathcal{F}(M)$, where $E_{\alpha,\beta}$  is a canonical basis
of $\mathcal{M}_d(\mathbb{R})$, $\mathcal{M}_d(\mathbb{R})$ is the $d\times d$ matrix space on $\mathbb{R}$,
and $l_u:Gl_d(\mathbb{R})\rightarrow \mathcal{F}(M)$ is the left multiple operator from the general linear group $Gl_d(\mathbb{R})$ to $\mathcal{F}(M)$, i.e. $l_u\exp(E_{\alpha,\beta})=u\exp(E_{\alpha,\beta})$.

Let $B_t:=(B_t^1,B_t^2,\cdots,B_t^d)$ be a $\mathbb{R}^d$-valued Brownian motion on  a complete
filtered probability space $(\Omega,\{\mathscr{F}_t\}_{t\geq 0}, \mathbb{P})$.
To construct a $L_t$-diffusion process,
we first need to construct the corresponding horizontal diffusion process
 by solving the Stratonovich SDE
\begin{align}\label{SDE-u}
\begin{cases}
 d u_t=\sqrt{2}\dsum_{i=1}^{d}H_{i}^t(u_t)\circ d
B_t^{i}+H_{Z_t }^t(u_t)d t-\frac{1}{2}\dsum_{\alpha,
\beta=1}^{d}\mathcal{G}_{\alpha,\beta}(t, u_t)V_{\alpha \beta}(u_t)d t,\medskip\\
u_s\in \mathcal{O}_s(M),\ {\bf p}u_s=x,\ s\in [0,T_c),
\end{cases}
\end{align}
where $\mathcal{G}_{\alpha,\beta}(t,u_t):=\partial_tg_t(u_te_{\alpha}, u_te_{\beta})$.  Similarly as explained in \cite{ACT}, the last term is essential to ensure $u_t\in \mathcal{O}_t(M)$.
 Since $\{H_{Z_{t} }^{t}\}_{t\in [0,T_c)}$ is $C^{1,\infty}$-smooth, the equation has a
unique solution up to its life time $\zeta:=\dlim_{n\rightarrow \infty}\zeta_n$, where
\begin{align}\label{zeta-n}
\zeta_n:=\inf\{t\in [s,T_c):\rho_t(\mathbf{p}u_0, \mathbf{p}u_t)\geq n\}, \ n\geq 1,\ \ \inf\varnothing:=T_c,
\end{align}
and $\rho_t$ stands for  the Riemannian distance induced by the metric $g_t$. Let $X_t^{(s,x)}=\mathbf{p}u_t$. Then $X_t^{(s,x)}$ solves the
equation
$$d X_t^{(s,x)}=\sqrt{2} u_t\circ d B_t+Z_t(X_t^{(s,x)})d t,\  \ X_s^{(s,x)}=x:=\mathbf{p}u_s$$
up to the life time $\zeta$. By the It\^{o} formula, for any $f\in
C_0^2(M)$ and $t\in [s,T_c)$,
$$f(X_t^{(s,x)})-f(x)-\int_s^tL_rf(X^{(s,x)}_r)d r=\sqrt{2}\int_s^t\l<u_r^{-1}\nabla^rf(X^{(s,x)}_r), d B_r\r>$$
is a martingale up to $\zeta$, where $\l<\cdot,\cdot\r>$ is the inner product on $\mathbb{R}^d$; that is, $X_t^{(s,x)}$ is the diffusion process
generated by $L_t$. When
$Z_t =0$,  the process $Y_t:=X^{(s,x)}_{t/2}$, $t\in [2s, 2T_c)$ is generated by
$\frac{1}{2}\Delta_{t/2} $ and is known as the $g_{t/2}$-Brownian motion. When $s=0$, we simply write $X_t^{(0,x)}$ as $X_t^{x}$ or $X_t$ without confusion.

  Throughout this article, we assume the diffusion process $X_t$ generated by $L_t$ is non-explosive before time $T_c$ (see  \cite{KR} for sufficient conditions to ensure the non-explosive).
Then this process  gives rise to an inhomogeneous Markov  semigroup $\{P_{s,t}\}_{0\leq
s\leq t< T_c}$ on $\mathcal{B}_b(M)$:
$$P_{s,t}f(x):=\mathbb{E}(f(X_t^{(s,x)})),\ x\in M,\ \  f\in \mathcal{B}_b(M),$$
 which is called the
diffusion semigroup generated by $L_t$. Here and in what follows,
$\mathbb{E}$ stands for the expectation
 taken for the underlying process. Moreover,  the Markov semigroup $\{P_{s,t}\}_{0\leq s\leq t<T_c}$ enjoys the following properties.
\begin{proposition}\label{0t1}
The following properties hold true.
\begin{itemize}
 \item [(i)] For any $0\leq s\leq t<T_c$, $f\in \mathcal{B}_b(M)$ and $x\in M$, there exists a unique probability measure $p_{s,t}(x,dy)$ such that
   $$P_{s,t}f(x)=\int_M f(y)\,p_{s,t}(x,d y).$$
  \item [(ii)] The measure $p_{s,t}(x, dy)$ is equivalent to the volume measure $\mu_t$ with respect to the metric $g_t$, that is
  \begin{equation}\label{fundamental-p}
 p_{s,t}(x, dy)=p(s,x;t,y)\mu_t(dy),
  \end{equation}
 where $p(s,x;t,y)$ is a fundamental solution to the following equation:
\begin{align*}
 \begin{cases}
      \frac{\partial}{\partial s}p(\cdot,x; t,y)(s)=-L_sp(s,\cdot;t,y)(x);  &\\
      \lim _{t\downarrow s}p(s,x; t,\cdot)=\delta_{x}(\cdot). &
   \end{cases}
\end{align*}
\item [(iii)]
For any $f\in \mathcal{B}_b(M)$ and $0\leq s\leq t<T_c$,
$$P_{s,t}f(x)=\int_M f(y) p(s,x;t,y)\mu_t(d y).$$
  Moreover, the backward Kolmogorov equation
\begin{align}\label{Km2}
\frac{d}{d s}P_{s,t}f=-L_sP_{s,t}f
\end{align}
holds for all $ 0\leq s\leq t< T_c$.
 \item [(iv)]
  If  $f\in C^{1,2}([0,T_c)\times M)$ such that $\|(L_t+\partial_t)f\|_{\infty}:=\sup_{x\in M}|(L_t+\partial_t)f|(x)$ is locally
bounded with respect to $t$ in $[0,T_c)$, then the forward Kolmogorov equation
\begin{align}\label{Km1}
\frac{d}{d t}P_{s,t}f(t,x)=P_{s,t}(L_t+\partial_t)f(t,x)
\end{align}
holds for all $0\leq s\leq t< T_c$.
\end{itemize}

\end{proposition}
\begin{proof}
(a). Let $X_t$ be a $L_t$-diffusion process. By the Markov property,  for $0\leq s< t<T_c$ and $x\in M$, let $p_{s,t}(x,\cdot)=\mathcal{L}(X_t|X_s=x)$, the law of $X_t$ conditional $X_s=x$. Then
$$P_{s,t}f(x)=\mathbb{E}(f(X_t^{(s,x)}))=\int_Mf(y)p_{s,t}(x, dy)$$
for any $f\in \mathcal{B}_b(M)$.

(b). First,  suppose $M$ is compact.   Then it is easy for us to see from \cite{Gu} that by replacing the Laplacian operator $\Delta_t$  with $\Delta_t+Z_t$, and  repeating the same argument
as in the proof of \cite[Theorem 2.1]{Gu},  the existence of a fundamental solution on compact manifolds can be derived similarly.  Then, for general case,
given an open set $\Omega \subset M$, one can treat $\Omega$ as a manifold itself. Let us denote by
$p^{\Omega}$ the heat kernel of $\Omega$. Minimality of the heat kernel implies that $p^{\Omega}$ vanishes on
the boundary $\partial \Omega$, at least if $\partial \Omega$ is smooth. This implies, in turn, that $p^{\Omega}$ increases
on enlarging of $\Omega$ (see \cite{Do83}).
The way the global heat kernel $p$ is constructed in $M$  is the following: one first
defines $p^{\Omega}$ for precompact sets  $\Omega$ and then lets
$p:= \lim_{k\rightarrow \infty} p^{\Omega_k}$
where $\{\Omega_k\}$ is an increasing sequence of precompact open sets with smooth boundaries,
which exhaust  $M$. The differentiability conditions on $p(s, x; t, y)$ follow from those on fundamental solutions
in the fixed metric case, together with the $C^{1,\infty}$-smoothness of the metric.

(c).
Let $u(s,x)=\int_M p(s,x; t,y)f(y)\mu_t(dy)$. Then it  is easy to see that $u$ is a solution to the following heat equation:
\begin{align*}
  \begin{cases}
\frac{\partial}{\partial s}u(s,x)=-L_su(s,\cdot)(x);\\
u(t,x)=f(x).
\end{cases}
\end{align*}
Thus by the Feymann-Kac formula, we have
 $$\int_M f(y)p_{s,t}(x,dy)=P_{s,t}f(x)=u(s,x)=\int_M p(s,x; t,y)f(y) \mu_t(dy).$$

(d).  By using the It\^{o} formula,
$$ f(t,X_t^{(s,x)})-f(s,x)= M_t+ \int _s^t(L_r+\partial_r)f(r,X^{(s,x)}_r) d r,$$
with $M_t$ being a true martingale due to the boundedness of $(L_r+\partial_r)f$ on each interval $[s,t]$,   we prove the result by taking expectations on both sides of the above equation.
\end{proof}

\begin{remark}\label{rem-fun}
In \cite{Gu}, the author proved the existence of fundamental solution by using the parametrix method introduced by E. Levi.
It is easy to see from   \cite{Gu, ll62} that the  fundamental solution also exists for
the equation  by replacing  $L=\Delta_t+\frac{\partial}{\partial t}$  with a locally uniformly parabolic operator on a manifolds $M$, which is  written in locally coordinates as
$$Lu=\sum_{i,j} a^{i,j}(x,t) \frac{\partial^2 u}{\partial x^i\partial x^j}+\sum_{i}b^{i}(x,t)\frac{\partial u}{\partial x^i}+\frac{\partial u}{\partial t},$$
and all the coefficients involved are at least $C^1$.
\end{remark}

Given $T\in (0,T_c)$,  we deduce from   (\ref{Km2}) that $P_{s,T}f, s\in [0,T]$ is the
solution to the  following backward heat equation
\begin{align}\label{Chy0}
\begin{cases}
\partial_s u(\cdot,x)(s)=-L_su(s,\cdot)(x),\ s\in [0,T];\\
u(T,x)=f(x).
\end{cases}\end{align}
Indeed, the theory presented here is meant to be applied to our familiar forward heat equation by a time reversal. More precisely, let $(X_t^T)_{t\in [0,T]}$ be a
$L_{(T-t)}$-diffusion process, which is assumed to be non-explosive before time $T$, and  $\{\overline{P}_{s,t}\}_{0\leq s\leq
t\leq T}$ be the associated semigroup.  Then $\overline{P}_{T-t,T}f,\ t\in [0,T]$ solves the equation
\begin{align}\label{Chy}
\begin{cases}
\partial_t u(\cdot,x)(t)=L_tu(t,\cdot)(x),  \  \  t\in [0,T],\\
u(0,x)=f(x).
\end{cases}\end{align}
We would like to indicate that in \cite{ACT}, derivative formulae
and gradient estimates of the semigroup $\{\overline{P}_{s,t}\}_{0\leq s\leq
t\leq T}$ have been investigated   by using the  $g_{(T-t)}$-Brownian motion under Ricci flow.
\medskip

\section{Derivative formulas and their applications}
\subsection{Derivation formulas}
In this subsection, we   establish derivative formulas on the local and whole manifold respectively, which is then applied to gradient estimates, dimensional-free Harnack inequalities and other functional inequalities of the semigroup.

For $u\in \mathcal{O}_t(M)$,  the lift operators $\mathcal{R}_t^{Z}(u), \  \mathcal{G}_t(u)\in \mathbb{R}^d\otimes\mathbb{R}^d$ are defined by
\begin{align*}
&\mathcal{R}^{Z}_t(u)(a,b)=\l<\mathcal{R}^{Z}_t(u)a,b\r>=\mathcal{R}_t^{Z}(ua,
u b)\ \ \mbox{and} \  \
 \mathcal{G}_t(u)(a,b)=\partial_tg_t(ua,ub),
\end{align*}
where $ a,b\in \mathbb{R}^d$.
Now, let us  introduce the $\mathbb{R}^d\otimes \mathbb{R}^d$-valued
process $\{Q_{s,t}\}_{0<s\leq t<T_c}$, which solves the ODE: for $a,b \in \mathbb{R}^d$,
\begin{align}\label{damp}
\begin{cases}
&\frac{d \l<Q_{s,t}a,b\r>}{d
t}=-\l<\mathcal{R}^Z_t(u_t)Q_{s,t}a,b\r>,\\
&Q_{s,s}=I,
\end{cases}\end{align}
 where $u_t$ is the horizontal $L_t$-diffusion process of $X_t$.
 When $s=0$, we simply write $Q_t:=Q_{0,t}$.

 If there exists $K\in C([0,T_c)\times M)$  such that
$\mathcal{R}_t^{Z}\geq
K(t,\cdot) $ for each $t\in [0,T_c)$,   then, from \eqref{damp}, it follows that
$$\|Q_{s,t}\|\leq \exp{\l[-\int_s^tK(r, X_r)d r\r]},$$
where $\|\cdot\|$ is the operator norm on $\mathbb{R}^d$.We now introduce a local version of derivative formula of $P_{s,t}$.
When it reduces to the fixed metric case, it looks more like that  given by  A.Thalmaier \cite{Thalmaier}.
\begin{theorem}\label{1t1}
For $0\leq s<t<T_c$, let $x\in M$ and $D$ be a  compact domain in
$[s,t]\times M$ such that $(s,x)\in D^{\circ}$, the interior of $D$.
Let $\tau_D=\inf\{r\in (s,T_c): (r,X^{(s,x)}_{r})\notin D\}$ and
$F\in C^{1,2}([s,t]\times D)$ satisfy the heat equation
\begin{equation}\label{add-V2-1}
\partial_r F(\cdot,x)(r)=-L_rF(r,\cdot)(x),
\end{equation}
for all $(r,x)\in [s, t]\times D$. Then for any adapted absolutely continuous $\mathbb{R}_+$-valued
process $h$ such that $h(s)=0$ and $h(r)=1$ for all $r\geq
{t}\wedge{\tau_D}$, and $\mathbb{E}^{(s,x)}(\int_s^th'(r)^2d
r)^{\alpha}<\infty$ for some $\alpha >\frac{1}{2}$, we have
\begin{align}\label{D-formula}
(u_{s})^{-1}\nabla^sF(s,\cdot)(x)=\frac{1}{\sqrt{2}}\mathbb{E} ^{(s,x)}\l\{F({t\wedge \tau_{D}}
, X_{t\wedge \tau_D})\int_s^th'(r)Q^*_{s,r}d B_{r}
\r\},
\end{align}
where $Q^*_{s,r}$ is the transpose of $Q_{s,r}$.
\end{theorem}
\begin{proof}
Without loss of generality, we assume $s=0$ and simply drop the upper script $x$ in $X_t^x$. Since $Z$ is a $C^{1,\infty}$ vector field and $F$ is a classic solution of \eqref{add-V2-1}, then it is easy to see that $F\in C^{1,3}([s,t]\times D)$.  For $u\in \mathcal{O}_r(M)$ and $r\in [0,T_c)$,
let $G(u,r)=u^{-1}\nabla^rF_r({\bf p}u)$.
Then according to \eqref{add-V2-1}, we have
\begin{align*}
\frac{\partial}{\partial r}G(u,r)=\sum_{i=1}^d\frac{\partial}{\partial r}ue_iF_r({\bf p}u)e_i=-\sum_{i=1}^due_i(L_rF_r)({\bf p} u)e_i=-u^{-1}\nabla^{r}L_rF_r({\bf p}u).
\end{align*}
By this and the Bochner-Weitzenb\"{o}ck formula, we see that for $r\in [0,t]$,
\begin{align}\label{2.1}
\frac{\partial}{\partial r}G(u,r)=-L_{\mathcal{O}_r(M)}G(\cdot,r)(u)+\l(\mathcal{R}_r^Z(u)+\frac{1}{2}\mathcal{G}_r(u)\r)G(u,r),
\end{align}
where $L_{\mathcal{O}_r(M)}:=\Delta_{\mathcal{O}_r(M)}+H_{Z_r}^r$ and  $\Delta_{\mathcal{O}_r(M)}$ is the
horizontal Laplacian operator of $\Delta_r$.
On the other hand, noting that $u_r$ is the solution to  \eqref{SDE-u}, and
using the It\^{o} formula,  one obtains that for fixed $t_0\in [0,t]$,
\begin{align}\label{Ieq}
d G(u_r, t_0)=&d M_r+L_{\mathcal{O}_r{(M)}}G(\cdot,
t_0)(u_r)d r-\frac{1}{2}\sum_{\alpha,\beta}\mathcal{G}_{\alpha,\beta}(t_0,u_r)V_{\alpha,\beta}(u_r)G(\cdot,
t_0)(u_r)d r,
\end{align}
where $d M_r:=\sqrt{2}H^r_{u_rd B_r}G(\cdot,
t_0)(u_r).$
Moreover, the last  term above on the rightmost hand satisfies that
\begin{align*}
-\frac{1}{2}\sum_{\alpha,\beta}&\mathcal{G}_{\alpha,\beta}(t_0,u_r)V_{\alpha,\beta}(u_r)G(\cdot,
t_0)(u_r)\\
&=-\frac{1}{2}\sum_{\alpha,\beta}\mathcal{G}_{\alpha,\beta}(t_0,u_r)(u_re_{\beta}F_r)(X_r)\cdot e_{\alpha}\\
&=-\frac{1}{2}\mathcal{G}_{t_0}(u_r)G(u_r,
t_0).
\end{align*}
 With this, \eqref{Ieq} and (\ref{2.1}), we conclude that
\begin{align}\label{2.4}
d \l<\nabla^rF_r(X_r),u_rQ_ra\r>_r=d\l<G(u_r,r), Q_ra\r>=\sqrt{2}\mathrm{Hess}^r_{F_r}(u_{r}\mathrm{d}B_{r},u_rQ_ra)(X_r),
\end{align}
which implies that $\l<\nabla^rF_r(X_r),u_rQ_ra\r>_r$ is a local martingale.

On the other hand, by the It\^{o} formula, one has
$$d F(r,X_r)=\sqrt{2}\l<\nabla^rF_r(X_r), u_{r}d B_{r}\r>_{r},$$
and
\begin{align*}
F(t\wedge\tau_D,
X_{t\wedge\tau_D})&=F(0,x)+\sqrt{2}\int^{t\wedge\tau_D}_{0}\l<\nabla^rF_r(X_{r}),
u_{r}d B_{r}\r>_{r}.\end{align*}
Therefore, noting that $h'(r)=0$ for $r\geq t\wedge \tau_D$, we have
\begin{align}\label{e:local-eq}
&\mathbb{E}^x\l\{F({t\wedge\tau_D},
X_{t\wedge\tau_ D})\int_{0}^t\l<h'(r)Q_{r}a,
d B_{r}\r>\r\}\nonumber\\
=&\mathbb{E}^x\l\{\l(F(0,x)+\sqrt{2}\int_{0}^{t\wedge\tau_D}\l<\nabla^rF_{r}(X_{r}), u_{r}d
B_{r}\r>_{r}\r)\int_{0}^t\l<h'(r)Q_{r}a,
d B_{r}\r>\r\}.
\end{align}
As  $\int_{0}^t\l<h'(r)Q_{r}a,
d B_{r}\r>$ is a true martingale, we then have
\begin{align*}
&\frac{1}{\sqrt{2}}\mathbb{E}^x\l\{F({t\wedge\tau_D},
X_{t\wedge\tau_ D})\int_{0}^t\l<h'(r)Q_{r}a,
d B_{r}\r>\r\}\nonumber\\
=&\,\mathbb{E}^x\l\{\int_{0}^{t}\l<\nabla^rF_{r}(X_{r}), u_rQ_ra\r>_{r}(h-1)'(r)d r\r\}\nonumber\\
=&\,\mathbb{E}^x\l\{\l[\l<\nabla^rF_{r}(X_{r}),
 u_rQ_ra\r>_{r}\cdot(h-1)(r)\r]\big{|}_0^{t}\r\}-\mathbb{E}^x\int_0^{t}(h-1)(r)d
\l<\nabla^rF_{r}(X_{r}),  u_rQ_ra\r>_{r}\nonumber\\
=&\,\l<\nabla^0F_{0}(x), u_0a\r>_{0},
\end{align*}
where the last step follows from \eqref{2.4} that   $\int_0^t(h-1)(r)d
\l<\nabla^{r}F_{r}(X_r), u_rQ_ra\r>_r$ is a  true martingale.
Now given arbitrary $s\in [0,T_c)$, repeating the above argument, we prove  \eqref{D-formula} and then finish the proof.
\end{proof}

Next, we introduce the following  globe derivative formula for $P_{s,t}$ without using
hitting time. For the corresponding result in the fixed metric case, we refer the readers to \cite{TW}.
Let  ${\rm Cut}_t(x)$ be the set of the $g_t$-cut-locus of $x$
on $M$.

\begin{theorem}\label{1t3}
Assume that for each $s\in
[0,T_c)$, $(L_s+\partial_s)\rho^2_s\leq c+h_1(s)+h_2(s)\rho^{2}_s $ holds outside $\mathrm{Cut}_s(o)$ for some constant
$c>0$ and some non-negative functions $h_1,h_2\in C([0,T_c))$. If
\begin{align}\label{add-eq1}
\mathcal{R}_{s}^{Z}\geq h_3(s)-16e^{-\int_0^s(h_2(r)+16)d r}\rho^2_s
\end{align}
holds for some $h_3\in C([0,T_c))$, then for all  $0\leq s\leq t<T_c$ and  $h\in C^1([s,t])$
satisfying $h(s)=0$ and  $h(t)=1$,
\begin{align}\label{1-0}
u_s^{-1}\nabla^{s}P_{s,t}f(x)&=\mathbb{E}^{(s,x)}\l\{Q_{s,t}^*u_t^{-1}\nabla^{t}f(X_t)\r\}=\frac{1}{\sqrt{2}}\mathbb{E}^{(s,x)}\l\{f(X_t)\int_s^th'(r)Q^*_{s,r}d
B_r\r\},
\end{align}
where $ x\in
M$ and  $f\in C^1(M)$ such that $f$ is constant outside a compact set. In particular, by taking $h(r)=\frac{(r-s)\wedge (t-s)}{t-s}$,
it holds
$$u_s^{-1}\nabla^sP_{s,t}f(x)=\frac{1}{\sqrt{2}(t-s)}\mathbb{E}^{(s,x)}\l\{f(X_t)\int_s^tQ^*_{s,r}d B_r\r\}.
$$

\end{theorem}
\begin{proof}
We again assume $s=0$.  By the It\^{o} formula (see \cite[Theorem 2]{Ku}),
$$d \rho^2_r(X_r)\leq 2\sqrt{2}\rho_r(X_r)d b_r+\l(c+h_1(r)+h_2(r)\rho^2_r(X_r)\r)d r$$
holds for some one-dimensional Brownian motion $b_t$. Let
$\lambda(r)=\int_0^r(h_2(s)+16)d s.$
Then, we have
\begin{align*}
d \l[e^{-\lambda(r)}\rho^2_r(X_r)\r] \leq&
e^{-\lambda(r)}\l[2\sqrt{2}\rho_r(X_r)d
b_r+\l(c+h_1(r)+h_2(r)\rho^2_r(X_r)\r)d
r\r]\\
&-e^{-\lambda(r)}(h_2(r)+16)\rho^2_r(X_r)d r\\
=&2\sqrt{2}e^{-\lambda(r)}\rho_r(X_r)d
b_r-16e^{-\lambda(r)}\rho^2_r(X_r)d r +(c+h_1(r))e^{-\lambda(r)}d r.
\end{align*}
Thus, letting $C(t,x)=e^{\rho_0^2(x)+ct+\int_0^th_1(s)d
s}$, we obtain
\begin{align*}
\mathbb{E}^x\exp{\l\{16\int_0^{t\wedge\zeta_n}\rho^2_r(X_r)e^{-\lambda(r)}d
r\r\}}  \leq
&\mathbb{E}^x\exp{\l\{2\sqrt{2}\int_0^{t\wedge\zeta_n}\rho_r(X_r)e^{-\lambda(r)}d
b_r\r\}}\cdot C(t,x)\\
\leq&\mathbb{E}^x\exp{\l\{16\int_0^{t\wedge\zeta_n}\rho^2_r(X_r)e^{-2\lambda(r)}d
r\r\}^{1/2}}\cdot C(t,x),
\end{align*}
where $\zeta_n$ is defined as in \eqref{zeta-n} with $s=0$.  From this, it is easy to deduce that
$$\mathbb{E}^x\exp{\l\{16\int_0^{t\wedge\zeta_n}\rho^2_r(X_r)e^{-\lambda(r)}d
r\r\}}\leq C(t,x)^2.$$
Now letting $n\rightarrow\infty$, we arrive at
$$\mathbb{E}^x\exp{\l\{16\int_0^{t}\rho^2_r(X_r)e^{-\lambda(r)}d
r\r\}}\leq C(t,x)^2.$$
Combining this with \eqref{add-eq1}, and  letting
$K(r,x):=h_3(r)-16e^{-\int_0^r(h_2(s)+16)d s}\rho^2_r,$
we deduce that
$\mathcal{R}_r^{Z}\geq K_r,\
 r\in [0,t],$
and
\begin{align}\label{e:contral-curvature}
&\sup_{x\in {\bf K}}\mathbb{E}^xe^{\int_0^tK^-(r,X_r)d r} \leq\sup_{x\in
 {\bf K}}\mathbb{E}^x\exp{\l\{16\int_0^te^{-\int_0^r(h_2(u)+16)d
u}\rho^2_r(X_r)d r\r\}}<\infty,\
\end{align}
where ${\bf K}\subset M $ is a compact subset. Then,
 following the proof of   \cite[Theorems 8.5 and 9.1]{LXM}, it is  easy to derive from \eqref{e:contral-curvature} that  $\sup_{r\in
[0,t]}\||\nabla^{r}P_{r,t}f|_r\|_{\infty}<\infty$. Therefore,  the first equality in \eqref{1-0} for $s=0$ follows from  \eqref{2.4}
  by taking $F(r,x)=P_{r,t}f(x)$, $r\in [0,t]$.

   Next, due to the first equality in \eqref{1-0} for $s=0$, we know that
  for any $a\in \mathbb{R}^d$ with $\|a\|=1$,
$$\mathbb{E}^x\sup_{r\in[0,t]}|\l<\nabla^rP_{r,t}f(X_r),
u_rQ_ra\r>_r|\leq \sup_{r\in[0,t]}\||\nabla^{r}P_{r,t}
f|_r\|_{\infty}\mathbb{E}^xe^{\int_0^tK^-(r,X_r)d r}<\infty,$$
which implies that
$$\l<\nabla^rP_{r,t}f(X_r),u_rQ_ra\r>_r,\ r\in[0,t]$$
 is a uniformly integrable martingale. Hence (\ref{D-formula}) holds for $t$ in
 place of $t\wedge \tau_D$ and any $h\in C^1([0,t])$ with $h(0)=0$ and $h(t)=1$. Therefore, the second equality in \eqref{1-0} holds for $s=0$.
\end{proof}
If there exists a non-negative $\phi\in C([0,\infty))$ and $h\in C([0,T_c))$ such that $\mathcal{R}_t^Z\geq -h(t)\phi(\rho_t)$, then by a similar argument as in the proof of \cite[Lemma 9]{Ku}, there exists a non-increasing function $F$ satisfying $\lim_{r\rightarrow 0}rF(r)<\infty$ such that
\begin{align}\label{compare-rho}
(L_t+\partial_t)\rho_t(x)\leq F(\rho_t(x))+h(t)\int_0^{\rho_t(x)}\phi(s)d s+|Z_t(o)|_t.
\end{align}
Hence,  the assumptions in Theorem
\ref{1t3} are ensured by each of the following conditions:
\begin{itemize}
   \item [(A1)] there\ exists a non-negative function $C\in C([0,T_c))$ such that $\mathcal{R}^Z_t\geq -C(t)$ for all $t\in [0,T_c);$
  \item [(A2)] there\ exist two non-negative functions $C_1, C_2
\in C([0,T_c)),$ such that  for all $t\in [0,T_c)$,
 $${\rm Ric}_t\geq -C_1(t)(1+\rho_t^2)\ \
{\rm and} \  \ \partial_t\rho_t+ \l<Z_t ,\nabla^{t}\rho_t\r>_t\leq
C_2(t)(1+\rho_t).$$
\end{itemize}
\begin{remark}\label{add-rem-1}
If (A2) holds, then
\begin{align*}
(L_t+\partial_t)\rho_t^2=& 2\rho_t(L_t+\partial_t)\rho_t+2
=  2\rho_t(\Delta_t+\partial_t+Z_t)\rho_t+2\\
\leq & 2\rho_t \Delta_t\rho_t +2C_2(t)(\rho_t+\rho_t^2)+2\\
\leq &2\rho_t\sqrt{(d-1)C_1(t)(1+\rho_t^2)}\coth\l(\sqrt{C_1(t)(1+\rho_t^2)/(d-1)}\rho_t\r)+2C_2(t)(\rho_t+\rho_t^2)+2.
\end{align*}
Combining this  with  the inequality $\coth(s)\leq 1+s^{-1}$, we know that under (A2), the conditions in Theorem \ref{1t3} are satisfied.
\end{remark}

\subsection{Gradient estimates}
In \cite{TW},  a local gradient estimate is given by using the local geometry of $M$. Inspired from  this,  local  gradient estimates of $P_{s,t}f$ are studied by using local version of derivative formula, and the corresponding result is presented as follows.
\begin{theorem}\label{c1}
Let $\mathcal{R}_s^Z\geq K_s$ for
some $K\in C([0,T_c)\times M)$.   Given $x\in M$, let
$$\kappa_s(x)=\sup_{r\in [s,s+1]}\l(\sup_{B_r(x,1)}K(r,\cdot)^-+|Z_r|_r(x)\r).$$
Then there exists a constant $c>0$ such that for any $f\in \mathcal{B}_b(M)$,
\begin{align}\label{local-g-est}
|\nabla^sP_{s,t}f|_s(x)\leq\frac{\|f\|_{\infty}\exp{[c(1+\kappa_s(x))]}}{\sqrt{(t-s)\wedge1}}.
\end{align}
\end{theorem}
\begin{proof}
Without loss of generality, we  consider $s=0$ for
simplicity. By the semigroup property and the contraction of
$P_{s,t}$, it suffices to prove \eqref{local-g-est} for $0<t\leq 1\wedge T_c$.   It is easy to see that if we choose
  some explicit process $h$
 in  Theorem \ref{1t1}, then the associated gradient
estimate of $P_{s,t}f$ can be achieved by only using local geometry of the manifold. To this end,  for $x\in M$, let  $D=\{(r,y)\in [0,t]\times
M:\rho_r(x,y)\leq 1\}$. It is easy to see that $D$ is closed and hence compact, since
$\rho_r(x,y)$ is continuous in $(r,x,y)$ (see \cite[Lemma 2.5]{Ku3}).
  Let
$\varphi(t, X_t)=\cos(\pi\rho_t(x,X_t)/2)$. Let $X_0=x$ and
$$T(r)=\l(\int_0^r\varphi^{-2}(t,X_t)d t\r) \mathbf{1}_{\{r\leq \tau_D\}}+\infty {\bf 1}_{\{r>\tau_D\}},$$
where  $\tau_D$ is the first hitting time of $(r,X_r)$ to
$\partial D$. Define
$$\tau(r)=\inf\{u\geq 0: T(u)\geq r\},\ \ r\geq 0.$$ Then $T\circ\tau(r)=r$ for $r\leq \tau_D$.  Moreover,
$$\tau'(r)=\frac{1}{T'\circ\tau(r)}=\varphi^2(\tau{(r)}, X_{\tau{(r)}}),\ \ r\leq \tau_D.$$
Since $\varphi\leq 1$, we have
$\tau(r)\leq r$.
Define $h(r)=\frac{1}{t}\int_0^{r\wedge \tau(t)}\varphi^{-2}(u,X_u)d
u.$ Then, $h$ meets the requirement of Theorem \ref{1t1} and
\begin{align}\label{eq1}
\int_0^{\tau(t)}h'(r)^2d
r&=\frac{1}{t^2}\int_0^{\tau(t)}\varphi^{-4}(r,X_r)d
r=\frac{1}{t^2}\int_0^{\tau(t)}\varphi^{-2}(r,X_r)d
T(r)=\frac{1}{t^2}\int_0^t\varphi^{-2}(\tau(r),X_{\tau(r)})d r.
\end{align}
Moreover, let $v\in T_xM$ and $|v|_0=1$. By the definition of $Q_r$
and $\mathcal{R}_r^{Z}\geq -\kappa_0$  on
$D$, we have
$$|u_rQ_ru_0^{-1}v|_r\leq |v|_0e^{\kappa_0},\ \ r\leq \tau(t),\ \ 0< t\leq 1.$$
Combining this  with Theorem
\ref{1t1}, we have
\begin{align*}
  \l|\l<\nabla^0P_{0,t}f(x), v\r>_0\r|&\leq \frac{1}{\sqrt{2}} \|f\|_{\infty}e^{\kappa_0(x)}\l(\mathbb{E}^x\int_0^{\tau(t)}h'(r)^2d r\r)^{1/2}\\
&=\frac{1}{\sqrt{2}} \|f\|_{\infty}e^{\kappa_0(x)}\frac{1}{t^2}\int_0^t\varphi^{-2}(\tau(r),X_{\tau(r)})d r.
\end{align*}
 Thus, it suffices for us to estimate  the last term above. Before  this, we prove that $(\tau(r),X_{\tau(r)})$ is non-explosive on
$D$, i.e. the life time $\tau:=\inf\{r\geq 0, (\tau(r),X_{\tau(r)})\in \partial D\}$ satisfies $\tau=\infty$, a.e.

  For $n\geq 1$, let $\tau_n=\inf\{r: \varphi(\tau(r),X_{\tau(r)})\leq 1/n\}$.
 Note that $X_{\tau(r)}$ is generated by $\varphi^2L_{\tau(r)}$, then by \eqref{compare-rho}, there exists some constant $c>0$ such that
\begin{align*}
\varphi^2(L_{\tau(r)}+\partial_1)\varphi^{-1}&=-(L_{\tau(r)}+\partial_1)\varphi+2\varphi^{-1}|\nabla^{\tau(r)}\varphi|_{\tau(r)}^2\leq c\varphi^{-1},
\end{align*}
where $\partial_1$ denotes the derivative with respect to the first variable.
Hence, we have
\begin{align}\label{e:equ-3}
\mathbb{E}^x\varphi^{-1}(\tau(t\wedge \tau_n),X_{\tau(t\wedge \tau_n)})\leq \varphi^{-1}(0,x)e^{ct}=e^{ct},\qquad t\geq 0,\ n\geq 1.\end{align}
On the other hand, $\mathbb{E}^x\varphi^{-1}(\tau(t\wedge \tau_n),X_{\tau(t\wedge \tau_n)})\geq n\mathbb{P}(\tau_n<t)$. By this and \eqref{e:equ-3},  it is easy to see that
$$\mathbb{P}^x(\tau_n<t)\leq n^{-1}e^{ct}.$$
Letting $n\rightarrow\infty$, we have $\mathbb{P}^x(\tau<t)=0,\ t\geq 0$, i.e. $\mathbb{P}^x(\tau=\infty)=1$.

Next,  by using the It\^{o} formula,
\begin{align}\label{eq2}
d \varphi^{-2}(\tau(r), X_{\tau(r)})\leq d M_r + \l[\varphi^2\l(L_{
{\tau(r)}} +\partial_1\r)\varphi^{-2}\r](\tau(r),
X_{\tau(r)})d r
\end{align}
holds for some local martingale $M_r$. According to \eqref{compare-rho} and the
definition of $\kappa$, there exists a constant $c_1>0$ such that for all $ r\in [0, t]$,
$$\sin(\pi \rho_r(x,\cdot)/2)\l(L_r+{\partial_r}\r)\rho_r(x,\cdot)\leq c_1(1+\kappa_0(x))$$
holds on $D$. Thus, there exists a constant $c_2>0$ such that for all $r\in [0,t]$,
\begin{align*}
\varphi^{2}\l(L_r+{\partial_
r}\r)\varphi^{-2}&=-2\varphi^{-1}\l(L_r+{\partial_
r}\r)\varphi+6\varphi^{-2}|\nabla^r\varphi|_r^2\leq c_2(1+\kappa_0(x))\varphi^{-2}
\end{align*} holds on $D$.
Combining this with (\ref{eq1}) and (\ref{eq2}), we obtain that for  $t\in (0,1]$,
\begin{align}\label{eq31}
\mathbb{E}^x\int_0^{\tau(t)}h'(r)^2d r&=
\frac{1}{t^2}\int_0^t\mathbb{E}^x \varphi^{-2}(\tau(r), X_{\tau(r)})d
r\notag\\
&\leq \frac{1}{t^2}\int_0^t e^{c_2(1+\kappa_0(x))r}d r\leq \frac{c_3}{t}e^{c_3(1+\kappa_0(x))}
\end{align}
for some constant $c_3>0$.  From this, it follows that
$$\l|\l<\nabla^0P_{0,t}f(x), v\r>_0\r|\leq \|f\|_{\infty}e^{\kappa_0(x)}\l(\mathbb{E}\int_0^{\tau(t)}h'(r)^2d r\r)^{1/2}\leq\frac{\|f\|_{\infty}c_4e^{c_4(1+\kappa_0(x))}}{\sqrt{t}}$$
holds for some constant $c_4>0$ and all $t\in (0,1]$. This completes
the proof.
 \end{proof}
\begin{remark}\label{gradient-ineq}
We would like to indicate that recently, the author with X. Chen and J. Mao give local gradient estimates for some important geometric quantities
under Ricci flow, mean curvature flow and Yamabe flow by using some probabilistic method, see \cite{CCM} for more details.
\end{remark}

Next,   we aim to provide various equivalent semigroup inequalities for the lower bound curvature condition, i.e.
\begin{align}\label{CV}
\mathcal{R}^Z_t\geq K_t,\ \ \mbox{for  some}\ K\in C([0,T_c)\times M).
\end{align}
By using   the derivative formula, we have

\begin{theorem}\label{3t1}
Assume ${(A1)}$ or ${ (A2)}$ holds. Let $p\geq 1$
and $\tilde{p}=p\wedge 2$. Then for any $K\in C([0,T_c)\times M)$ and $t\in [0,T_c)$ such
that $K_t(x)^-/\rho^2_t(x)\rightarrow 0$ as
$\rho_t(x)\rightarrow \infty$, the following statements are
equivalent to each other.
\begin{enumerate}
  \item[$(i)$]  The curvature condition \eqref{CV} holds for the function $K$.
  \item[$(ii)$] For any $x\in M$,  $0\leq s\leq t<T_c$ and $f\in
 C_c^1(M)$,  $$|\nabla^sP_{s, t}f (x)|_s^p\leq
\mathbb{E}^{(s,x)}\l\{|\nabla^{t}f|_t^p(X_t)e^{-p\int_s^tK(r,X_{r})d
 r}\r\}.$$
  \item [$(iii)$] For any $0\leq s\leq t< T_c$, $x\in M$ and positive function $f\in C_c^1(M)$,
$$\frac{\tilde{p}[P_{s,t}f^2-(P_{s,t}f^{2/\tilde{p}})^{\tilde{p}}]}{4(\tilde{p}-1)}\leq \mathbb{E}^{(s,x)}\l\{|\nabla^{t}f|_t^2(X_t)\int_s^te^{-2\int_u^tK(r,X_{r})d r}d u\r\},$$
where when $p=1$, the inequality is understood as its limit as
$p\downarrow 1$:
\begin{align*}
P_{s,t}(f^2\log
f^2)(x)-(P_{s,t}f^2\log P_{s,t}f^2)(x)
\leq \  4
\mathbb{E}^{(s,x)}\l\{|\nabla^{t}f|^2_t(X_t)\int_s^te^{-2\int_u^tK(r,X_{r})d
r}d u\r\}.\end{align*}
  \item[$(iv)$] For any $0\leq s\leq t<T_c$, $x\in M$ and positive function $f\in C_c^1(M)$,
 \begin{align*}
 |\nabla^sP_{s,t}f|^2_s(x)
 \leq  \ \frac{[P_{s,t}f^{\tilde{p}}-(P_{s,t}f)^{\tilde{p}}](x)}{\tilde{p}(\tilde{p}-1)
 \int_s^t\l(\mathbb{E}^{(s,x)}\l\{(P_{u,t}f)^{2-\tilde{p}}(X_{u})e^{-2\int_s^uK(r, X_{r})d r}\r\}\r)^{-1}d
 u},\end{align*}
where when $p=1$, the inequality is understood as its limit as
$p\downarrow 1$:
\begin{align}\label{e-10}
|\nabla^sP_{s,t}f|^2_s(x)\leq \frac{[P_{s,t}(f\log f)-(P_{s,t}f)\log P_{s,t}f](x)}{\int_s^t\l(\mathbb{E}^{(s,x)}\l\{P_{u,t}f(X_{u})e^{-2\int_s^uK(r, X_{r})d r}\r\}\r)^{-1}d u}.\end{align}
\end{enumerate}
\end{theorem}
\begin{remark}\label{rem-add1}
 Consider a special case: $L=\Delta$ and $K=0$, i.e., $g_t$ is a super Ricci flow. Then Theorem \ref{3t1}(ii)-(iv) give equivalent characterizations for the super Ricci flow. It  is worth mentioning that very recently R. Haslhofer and A. Naber \cite{HN} use a sharp infinite dimensional gradient estimate to characterize solutions of the Ricci flow.
\end{remark}

 To prove this result, we first need to
characterize $\mathcal{R}^Z_t$ by
using the Taylor expansions first. When the metric is independent of $t$, the following results are essentially due to \cite{Bakry, Sturm}.
\begin{lemma}\label{1t2}
For $s\in [0,T_c)$ and $x\in M$, let $X\in T_x M$ with
$|X|_s=1$. Let $f\in C_0^{\infty}(M)$ such that $\nabla^{s}
f(x)=X$ and $\mathrm{Hess}^s_f(x)=0$,
and let $f_n=n+f$ for $n\geq 1$. Then,\\
$(i)$ for any $p>0$,
\begin{align*}
\mathcal{R}^Z_s(X,X)
&=\lim_{t\downarrow
s}\frac{P_{s,t}|\nabla^tf|^p_t(x)-|\nabla^{s}P_{s,t}f|^p_s(x)}{p(t-s)};
\end{align*}
 $(ii)$
 for any $p>1$,
 \begin{align}
\mathcal{R}^Z_s(X,X)
&=\lim_{n\rightarrow\infty}\lim_{t\downarrow s
 }\frac{1}{t-s}\l(\frac{p\{P_{s,t}f_n^2-(P_{s,t}f_n^{\frac{2}{p}})^p\}}{4(p-1)(t-s)}-|\nabla^{s}P_{s,t}f_n|_s^2\r)(x)\nonumber\\
 &=\lim_{n\rightarrow\infty}\lim_{t\downarrow s}\frac{1}{t-s}\l(P_{s,t}|\nabla^tf|^2_t-\frac{p\{P_{s,t}f_n^2-(P_{s,t}f_n^{\frac{2}{p}})^p\}}{4(p-1)(t-s)}\r)(x);\label{2-2}
\end{align}
$(iii)$ $\mathcal{R}^Z_s(X,X)$
is equal to each of the following limits:
\begin{align*}
&\lim_{n\rightarrow\infty}\lim_{t\downarrow
s}\frac{1}{(t-s)^2}\l\{(P_{s,t}f_n)\l[P_{s,t}(f_n\log
f_n)-(P_{s,t}f_n)\log
P_{s,t}f_n\r]-(t-s)|\nabla^{s}P_{s,t}f|_s^2\r\}(x);\\ 
&\lim_{n\rightarrow\infty}\lim_{t\downarrow
s}\frac{1}{4(t-s)^2}\l\{4(t-s)P_{s,t}|\nabla^tf|^2_t+(P_{s,t}f_n^2)\log
P_{s,t}f_n^2-P_{s,t}{f_n^2\log f_n^2}\r\}(x).
\end{align*}
\end{lemma}
\begin{proof}
(a). Without loss of generality, we only prove for $s=0$.
Since $\nabla^0f=X$ and ${\rm
Hess}_f^0(x)=0$, by the Bochner-Weitzenb\"{o}ck formula, we have
$$\Gamma_2^0(f,f)(x):=\frac{1}{2}L_0|\nabla^0f|^2_0(x)-\l<\nabla^0f, \nabla^0L_0f\r>_0(x)={\rm Ric}_0(X,X)-\l<\nabla_X^0{Z_0},X\r>_0.$$
Thus, the first assertion follows from
  the Taylor expansions at point $x$ (we drop $x$ below for simplicity):
\begin{align}\label{e-equ1}
P_{0,t}|\nabla^tf|_t^p
=|\nabla^0f|_0^p+\bigg(\frac{p}{2}|\nabla^0f|_0^{p-2}L_0|\nabla^0f|^2_0
-\frac{p}{2}|\nabla^0f|^{p-2}_0\mathcal
\partial_tg_t|_{t=0}(\nabla^0f, \nabla^0f)\bigg)\,t+{\rm o}(t),
\end{align}
and
\begin{align}\label{e-equ2}
|\nabla^0P_{0,t}f|^p_0=|\nabla^0f|_0^p+pt|\nabla^0f|^{p-2}_0\l<\nabla^0L_0f, \nabla^0f\r>_0+{\rm o}(t),
\end{align}
where in  the first equality we use the following equality,
$$\partial_t|\nabla^tf|^2_t=-\partial_tg_t(\nabla^tf,\nabla^tf).$$

(b). Let $f_n=n+f$, which is positive for large $n$. Then, for
small $t>0$ and large $n$,
\begin{align*}
P_{0,t}f_n^2-(P_{0,t}f_n^{2/p})^{p}
&=\frac{8(p-1)t^2}{p}\l<\nabla^0f,\nabla^0L_0f\r>_0-\frac{2(p-1)t^2}{p}\partial_tg_t|_{t=0}(\nabla^0f,\nabla^0f)\\
&\quad +
\frac{4(p-1)t}{p}|\nabla^0f|^2_0+\frac{4(p-1)t^2}{p}\Gamma^0_2(f,f)+t^2{\rm O}(n^{-1})+{\rm o}(t^2).
\end{align*}
 By this and  \eqref{e-equ2} for $p=2$, we prove the first equality  in \eqref{2-2}.
 Similarly, the second equality follows from \eqref{e-equ1} for $p=2$.

(c). The equalities in (iii) can be proved by combining
(\ref{e-equ1}) and (\ref{e-equ2}) for $p=2$ with the following two asymptotic formulae
respectively.
\begin{align*}
&(P_{0,t}f_n)\{P_{0,t}(f_n\log f_n)-(P_{0,t}f_n)\log P_{0,t}f_n\}\\
=&
t|\nabla^0f|^2_0+t^2\Gamma_2^0(f,f)+2t^2\l<\nabla^0f,\nabla^0L_0f\r>_0-\frac{1}{2}t^2\partial_tg_t|_{t=0}(\nabla^0f,\nabla^0f)+t^2{\rm O}(n^{-2})+{\rm o}(t^2);\\ \vspace{0.3cm}
&(P_{0,t}f_n^2)\log P_{0,t}f_n^2-P_{0,t}(f_n^2\log f_n^2)\\
=&-4t|\nabla^0f|^2_0-4t^2\l<\nabla^0L_0f,\nabla^0f\r>_0+2t^2\partial_tg_t|_{t=0}(\nabla^0f,\nabla^0f)-2t^2L_0|\nabla^0f|^2_0+{\rm o}(t^2)+t^2{\rm O}(n^{-1}).
\end{align*}
\end{proof}
\begin{proof}[Proof of Theorem \ref{3t1}]
According to the proof of Theorem \ref{1t3},
$\mathbb{E}^{(s,x)}\exp({p\int_s^t}K^-(s, X_s)d s)<\infty$ holds for any
$p>0, \ 0\leq s\leq t<T_c$ and $x\in M$.  So,  we obtain (i) of Lemma
\ref{1t2} by applying (ii) to $f\in C_0^{\infty}(M)$ such
that ${\rm Hess}^{s}_f(x)=0$ or applying (iii) to $n+f$
in place of $f$, or applying (iv) to $(f+n)^{2/p}$ when $p>1$ (resp.
$f+n$  when $p=1$) in place of $f$. Thus, it
suffices to show that (i) implies (ii)--(iv).

First,  if $\mathcal{R}^Z_t\geq
K_t,\ t\in [0,T_c)$, then by the first equality in (\ref{1-0})
and (\ref{damp}), we have
\begin{align*}
|\nabla^sP_{s,t}f|_s(x)&\leq \mathbb{E}^{(s,x)}\l\{|\nabla
^tf|_t(X_t)\exp{\l[-\int_s^tK(u, X_u)d
u\r]}\r\}\\
&\leq
\mathbb{E}^{(s,x)}\l\{|\nabla^{t}f|^p_t(X_t)\exp{\l[-p\int_s^tK(u,
X_u)d u\r]}\r\}^{1/p},
\end{align*}
thus, (ii) holds.

 To prove (iii) and (iv), let $p\in (1,2]$. 
 Without loss of generality, we only prove for $s=0$. By  Theorem
 \ref{0t1} and  (ii) with $p=1$,  for $0<u<\zeta_n$  ($\zeta_n$ is defined as in \eqref{zeta-n} with $s=0$),
\begin{align*}
d (P_{u,t}&f^{2/p})^p(X_u)\\
&=d M_u+(L_u+\partial_u)(P_{u,t}f^{2/p})^p(X_u)du\\
&=dM_u+p(p-1)(P_{u,t}f^{2/p})^{p-2}(X_u)|\nabla^{u}P_{u,t}f^{2/p}|^2_u(X_u)du\\
&\leq d M_u+p(p-1)(P_{u,t}f^{2/p})^{p-2}(X_u)\mathbb{E}^{(u,X_u)}\l\{\frac{2}{p}f^{\frac{2-p}{p}}(X_t)|\nabla^tf|_t(X_t){\rm e}^{-\int_u^tK(r,X_r)d r}\r\}^2d u\\
&\leq d M_u+\frac{4(p-1)}{p}(P_{u,t}f^{2/p})^{p-2}(P_{u,t}f^{2(2-p)/p})(X_u)\mathbb{E}^{(u,X_u)}\l(|\nabla^tf|_t^2(X_t)e^{-2\int_u^tK(r,X_r)dr}\r)du,
\end{align*}
 where $M_u$ is the local martingale part of $ (P_{u,t}f^{2/p})^p(X_u)$. Moreover,  since $2-p\in [0,1]$, by the Jensen inequality,
$$P_{u,t}f^{2(2-p)/p}\leq (P_{u,t}f^{2/p})^{2-p},$$
we  then arrive at
$$d(P_{u,t}f^{2/p})^p(X_u)\leq d M_u +\frac{4(p-1)}{p}\mathbb{E}^{(u,X_u)}\l(|\nabla^tf|_t^2(X_t)e^{-2\int_u^tK(r,X_r)dr}\r)du,  \ \ u<\zeta_n.$$
 By taking
integral over $[0,s\wedge \zeta_n]$, this implies
\begin{align*}
 \mathbb{E}^{x}(P_{s\wedge \zeta_n,t}f^{2/p})^p(X_{s\wedge \zeta_n})-(P_{0,t}f^{2/p})^p(x)\leq \int_0^{s\wedge\zeta_n}\frac{4(p-1)}{p}\mathbb{E}^x\l[|\nabla^t f|_t^2(X_t)\exp{\l(-2\int_u^t K(r,X_r) dr\r)}\r]d u.
\end{align*}
Letting $n\rightarrow \infty$, we obtain
\begin{align*}
 \frac{d}{d u}P_{0,u}(P_{u,t}f^{2/p})^p(x)\leq \frac{4(p-1)}{p}\mathbb{E}^x\l\{|\nabla^tf|_t^2(X_t)e^{-2\int_u^tK(r,X_r)d r}\r\}, \ \ \ u\in [0,t].
\end{align*}
This implies (iii) for $s=0$ by taking integral over $[0,t]$.

 Similarly,
$$d(P_{u,t}f(X_u))^p=d \tilde{M}_u+p(p-1)(P_{u,t}f)^{p-1}(X_u)|\nabla^u P_{u,t}f|^2_u(X_u)d u, \ \ 0<u<\zeta_n,$$
where $\tilde{M}_u$ is the local martingale part of $(P_{u,t}f(X_u))^p$,
which, together with (ii),  implies
\begin{align*}
\frac{d }{d
u}P_{0,u}(P_{u,t}f)^p&=p(p-1)P_{0,u}\{(P_{u,t}f)^{p-2}|\nabla^{u}P_{u,t}f|^2_u\}.
\end{align*}
Thus,
\begin{align*}
\frac{d }{d
u}P_{0,u}(P_{u,t}f)^p&\geq \frac{p(p-1)\l\{\mathbb{E}^x\l[|\nabla^{u}
P_{u,t}f|_u(X_u)e^{-\int_0^uK(r, X_r)d
r}\r]\r\}^2}{\mathbb{E}^x\l\{(P_{u,t}f)^{2-p}(X_u)e^{-2\int_0^uK(r,X_r)d
r}\r\}}\\
&\geq
\frac{p(p-1)|\nabla^0P_{0,t}f|_0^2}{\mathbb{E}^x\l\{(P_{u,t}f)^{2-p}(X_u)e^{-2\int_0^uK(r,X_r)d
r}\r\}}.
\end{align*}
Integrating over $[0,t]$, we prove (iv) for $s=0$.
\end{proof}

\subsection{Dimension-free Harnack inequalities}
The study of dimension-free Harnack inequalities  was initiated in \cite{W97}, which is applied to characterize some important properties of the underlying processes, see \cite{Wbook2}. Note that when the metric is fixed,
the following equivalence of (i) and (iv) are essentially due to \cite{WEq}.  In what follows, we simply write $p_{s,t}(x,y)=p(s,x;t,y)$.
\begin{theorem}\label{Har}
Let $p\in (1,\infty)$ and $K\in C([0,T_c))$. Then the following assertions
are equivalent to each other.
\begin{enumerate}
  \item [$(i)$] The curvature condition \eqref{CV}  holds for  the function $K$.
  \item [$(ii)$] For any $f\in \mathcal{B}_b^+(M)$ and $0\leq s\leq t<T_c$,
  $$(P_{s,t}f)^p(x)\leq P_{s,t}f^p(y)\exp{\l\{\frac{p}{4(p-1)}{\l[\int_s^te^{2\int_s^rK(u)d u}d
r\r]}^{-1}\rho_s^2(x,y)\r\}}.$$
  \item [$(iii)$]For any $f\in \mathcal{B}_b^+(M)$ with $f\geq 1$ and $0\leq s\leq t<T_c$,
  $$P_{s,t}\log f(x)\leq \log P_{s,t}f(y)+{\l[4\int_s^te^{2\int_s^rK(u)d u}d
r\r]}^{-1}\rho_s^2(x,y).$$
  \item [$(iv)$] For any $0\leq s\leq t<T_c$ and $x,y \in M$,
 \begin{align*}
 &\int_Mp_{s,t}(x,y)\l(\frac{p_{s,t}(x,y)}{p_{s,t}(y,z)}\r)^{\frac{1}{p-1}}\mu_{t}(d z)
 \leq \exp{\l\{\frac{p}{4(p-1)^2}{\l[\int_s^te^{2\int_s^rK(u)d u}d
r\r]}^{-1}\rho_s^2(x,y)\r\}}.\end{align*}
 \item [$(v)$] For any $0\leq s\leq t< T_c$ and $x,y\in M$,
 $$\int_M p_{s,t}(x,y)\log
 \frac{p_{s,t}(x,y)}{p_{s,t}(y,z)}\mu_{t}(d z)\leq
 {\l[4\int_s^te^{2\int_s^rK(u)d u}d
r\r]}^{-1}\rho_s^2(x,y).$$
\end{enumerate}
\end{theorem}
\begin{proof}
By \cite[Proposition 2.4]{W10},  (ii) and (iii) are equivalent to  (iv) and (v) respectively. Moreover, according to \cite[Corollary 1.4.3]{Wbook2}, we see that
(ii) implies (iii). Thus, it is sufficient to prove
that  ``(i)$\Rightarrow$(ii)" and ``(iii)$\Rightarrow$(i)".

(a)\ \emph{(i)\  implies\  (ii)}. We consider the case for
$s=0$. By approximation and the monotone class theorem, we may
assume that $f\in C^2(M), \inf f > 0$ and $f$ is constant outside
a compact set. Given $x \neq y$ and $t>0$, let $\gamma:
[0,t]\rightarrow M$ be the constant speed $g_0$-geodesic from $x$ to $y$ with
length $\rho_0(x,y)$. Let $\nu_s=\frac{d \gamma_s}{d s}$. Then
we have $|\nu_s|_0=\rho_0(x,y)/t$. Let
$$h(s)=\frac{t\int_0^se^{2\int_0^rK(u)d u}d r}{\int_0^te^{2\int_0^rK(u)d u}d r}.$$
Then $h(0)=0$ and  $h(t)=t$. Let $y_s=\gamma_{h(s)}$ and
$$\varphi(s)=\log P_{0,s}(P_{s,t}f)^p(y_s), \ s\in [0,t].$$
To get $\varphi'(s)$, we first see that
\begin{align*}
  d(P_{s,t}f(X_s))^p&=d M_s+(L_s+\partial_s)(P_{s,t}f)^p(X_s)d s\\
&=d M_s+ p(p-1)(P_{s,t}f)^{p-2}(X_s)|\nabla^sP_{s,t}f|_s^2(X_s)d s, \  \  \  s<\zeta_n,
\end{align*}
where $M_s$ is the local  martingale part of  $(P_{s,t}f(X_s))^p$, which implies that
$$\mathbb{E}^{x}(P_{s\wedge\zeta_n,t}f(X_{s\wedge\zeta_n}))^p-(P_{0,t}f)^p(x)=p(p-1)\mathbb{E}^x\int_0^{s\wedge \zeta_n}(P_{u,t}f)^{p-2}(X_u)|\nabla^uP_{u,t}f|_u^2(X_u)d u.$$
Moreover, due to Theorem \ref{3t1}(ii), it holds $|\nabla^uP_{u,t}f|_u\leq e^{-\int_u^tK(r)d
r}P_{u,t}|\nabla^{t} f|_t$. From this  and   $\inf f>0$,  we deduce   by  letting $n\rightarrow \infty$ that
\begin{equation}\label{Ito-eq}
\mathbb{E}^x(P_{s,t}f(X_{s}))^p-(P_{0,t}f)^p(x)=p(p-1)\int_0^{s}\mathbb{E}^x\l[(P_{u,t}f)^{p-2}(X_u)|\nabla^uP_{u,t}f|_u^2(X_u)\r]d u.
\end{equation}
By this and  the Kolmogorov equations, we obtain
\begin{align*}
\frac{d \varphi(s)}{d
s}=&\frac{1}{P_{0,s}(P_{s,t}f)^p}\Big\{p(p-1)P_{0,s}(P_{s,t}f)^p|\nabla^s\log
P_{s,t}f|^2_s+h'(s)\l<\nabla^0P_{0,s}(P_{s,t}f)^{p},\nu_s\r>_0\Big\}\\
\geq&\frac{p}{P_{0,s}(P_{s,t}f)^p}P_{0,s}\Big\{(P_{s,t}f)^p\Big((p-1)|\nabla^s\log
P_{s,t}f|^2_s-\frac{\rho_0(x,y)}{t}h'(s)e^{-\int_0^s K(u)d u}|\nabla^s\log
P_{s,t}f|_s\Big) \Big \} \\
\geq& \frac{-p\rho_0(x,y)^2 h'(s)^2e^{-2\int_0^sK(u)d u}}{4(p-1)t^2},
\end{align*}
for  $s\in [0,t]$. Since $h'(s)=\frac{te^{2\int_0^sK(u)d
u}}{\int_0^te^{2\int_0^rK(u)d u}d r} $, we arrive at
$$\frac{d \varphi(s)}{d
s}\geq \frac{-p\rho_0(x,y)^2e^{\int_0^s2K(u)d u}
}{4(p-1)(\int_0^te^{2\int_0^rK(u)d u}d r)^2},\qquad s\in [0,t].$$ By
integrating over $s$ from 0 and $t$, we complete the proof.

(b)  \emph{ (iii)\ implies \ (i)}. Suppose that $s=0$. Let $x\in M$ and $X\in T_xM$ be
fixed. For any $n\geq 1$, we may take $f\in C^{\infty}(M)$ such that
$f\geq n$, $f$ is constant outside a compact set, and
$$\nabla^0f(x)=X, \ \ {\rm Hess}_f^0(x)=0.$$
Taking $\gamma_t=\exp_x\l[-2t\nabla^0\log f(x)\r]$, we have
$\rho_0(x,\gamma_t)=2t|\nabla^0\log f|_0(x)$ for $t\in
[0,t_0]$, where $t_0$ is a positive constant such that $\rho_0(x, \gamma_t)<r$ and
$r>0$. By using (iv) with $y=\gamma_t$, we obtain
\begin{align}\label{log1}P_{0,t}(\log f)(x)\leq \log
P_{0,t}f(\gamma_t)+\frac{t^2|\nabla^0\log
f|^2_0(x)}{\int_0^te^{2\int_0^rK(u)d u}d r},\qquad \ t\in
[0,t_0].\end{align}
Since $L_0f\in C^2_0(M)$ and ${\rm Hess}_f^0(x)=0$ implies
$\nabla^0|\nabla^0 f|_0^2(x)=0$ at point $x$, 
  by Taylor's expansion, we have
\begin{align}\label{log2}
P_{0,t}(\log f)(x)=\log f(x)+t(f^{-1}L_0f-|\nabla^0\log
f|^2_0)(x)+\frac{t^2}{2}A+{\rm o}(t^2)
\end{align} for small $t>0$, where
\begin{align*}
A:=&\frac{L_0^2f}{f}-\frac{(L_0f)^2}{f^2}-\frac{2}{f^2}\l<\nabla^0L_0f,\nabla^0f\r>_0
-\frac{L_0|\nabla^0f|^2_0}{f^2}
+\frac{4|\nabla^0f|^2_0L_0f}{f^3}\\
&-\frac{6|\nabla^0f|_0^4}{f^4}+\frac{\partial_tg_t|_{t=0}(\nabla^0
f,\nabla^0 f)}{f^2}+\frac{1}{f}\frac{d L_tf}{d
t}\bigg|_{t=0}.
\end{align*}
 On the other hand, let
$N_t=P^0_{x,\gamma_t}\nabla^0\log f(x)$, where $P^0_{x,\gamma_t}$ is
the $g_0$-parallel displacement along the $g_0$-geodesic $\gamma: t\rightarrow
\gamma_t$. We have $\dot{\gamma}_t=-2N_t$ and
$\nabla^0_{\dot{\gamma}_t}N_t=0.$
 Thus, note that ${\rm Hess}^0_f(x)=0$,  the Taylor expansion of $\log P_{0,t}f(\gamma_t)$ at $x$ is the following
$$\log P_{0,t}f(\gamma_t)=\log f(x)+t(f^{-1}L_0f-2|\nabla^0 \log f|^2_0)(x)+\frac{t^2}{2}B+{\rm o}(t^2),$$
where
\begin{align*}
B:=&\frac{L_0^2
f}{f}-\frac{(L_0f)^2}{f^2}+\frac{4 L_0f|\nabla^0f|_0^2}{f^3}-\frac{4\l<\nabla^0L_0f,\nabla^0
f\r>_0}{f^2}-4\frac{|\nabla^0f|_0^4}{f^4}+\frac{1}{f}\frac{d L_tf}{d t}\bigg|_{t=0}.
\end{align*}
Combining this with (\ref{log1}) and (\ref{log2}), we arrive at
\begin{align*}
&\frac{1}{t}\l(1-\frac{t}{\int_0^te^{2\int_0^rK(u)d u}d
r}\r)|\nabla^0\log f|^2_0(x)\\
&\leq \frac{1}{2}\bigg(\frac{L_0|\nabla^0
f|_0^2-2\l<\nabla^0L_0f,\nabla^0
f\r>_0}{f^2}+\frac{2|\nabla^0f|_0^4}{f^4}+\frac{1}{f^2}\partial_tg_t|_{t=0}(\nabla^0f,\nabla^0f)\bigg)(x)+{\rm o}(1).
\end{align*}
Letting $t\rightarrow 0$, we obtain
\begin{align*}
\frac{1}{2}L_0|\nabla^0f|_0^2(x)-\l<\nabla^0L_0f,\nabla
f\r>_0(x)\geq
K(0)|\nabla^0f|^2_0(x)+\frac{1}{2}\partial_tg_t|_{t=0}(\nabla^0f,\nabla^0f)(x)-\frac{|\nabla^0
f|_0^4}{f^2}(x).
\end{align*}
By the Bochner-Weitzenb\"{o}ck formula,
it follows that for $ n\geq 1$,
$${\rm Ric}_0(X,X)-\l<\nabla_X^0{Z_0},X\r>_0\geq K(0)|X|^2_0+\frac{1}{2}\partial_tg_t|_{t=0}(X,X)-\frac{|X|^4_0}{n}.$$
This implies (i) for $s=0$ by letting $n\rightarrow \infty$.
\end{proof}

\begin{remark}\label{rem1}
 Let $p^{s,t}_{x,y}(z)=\frac{p_{s,t}(x,z)}{p_{s,t}(y,z)}$ for $x,y,z\in M$ and $0\leq s\leq t<T_c$.
According to \cite[Proposition\ 2.4]{W10}, we have the following statements, which are equivalent to  Theorem \ref{Har} (ii)(iii) respectively.
\begin{enumerate}
 \item [(ii')] For any $0\leq s\leq t<T_c$,  $p^{s,t}_{x,y}$ satisfies
 \begin{align*}
 P_{s,t}((p^{s,t}_{x,y})^{1/(\alpha-1)})(x)\leq \l\{\frac{p}{p-1}{\l[4\int_s^te^{2\int_s^rK(u)d u}d
r\r]}^{-1}\rho_s^2(x,y)\r\}^{1/(\alpha-1)},\ \
 x,y\in M.
 \end{align*}
 \item [(iii')]  For any $0\leq s\leq t<T_c$,  $p^{s,t}_{x,y}$ satisfies
 $$P_{s,t}\{\log p^{s,t}_{x,y}\}(x)\leq {\l[4\int_s^te^{2\int_s^rK(u)d u}d
r\r]}^{-1}\rho_s^2(x,y),\ \ x,y\in M.$$
 \end{enumerate}
\end{remark}

\subsection{Other functional inequalities}

 In \cite{BBG} functional inequalities of the following type are
shown to be useful on manifolds with a fixed metric. We now extend this type of results to our case.

\begin{theorem}\label{other-ineq}
Let $K\in C([0,T_c))$. The following assertions
are equivalent to each other.
\begin{enumerate}
  \item [$(i)$] The curvature condition \eqref{CV}  holds for the function $K$.

\item [$(ii)$] For any $0\leq s\leq r\leq t<T_c$ and $1<q_1\leq q_2$ such
that
\begin{align}\label{3e11}
\frac{q_2-1}{q_1-1}=\frac{\int_s^te^{2\int_s^uK(\tau)d \tau}d
u}{\int_s^re^{2\int_s^uK(\tau)d \tau}d u},
\end{align}
it holds
$$\{P_{s,r}(P_{r,t}f)^{q_2}\}^{\frac{1}{q_2}}\leq (P_{s,t}f^{q_1})^{\frac{1}{q_1}}, $$
for all positive function $f\in\mathcal{B}_b(M).$
\item [$(iii)$]  For any $0\leq s\leq r\leq t<T_c$ and $0<q_2\leq q_1$ or $q_2\leq
q_1<0$ such that $(\ref{3e11})$ is satisfied,
$$(P_{s,t}f^{q_1})^{\frac{1}{q_1}}\leq \{P_{s,r}(P_{r,t}f)^{q_2}\}^{\frac{1}{q_2}}$$
for all positive function $f\in\mathcal{B}_b(M).$
\end{enumerate}

\end{theorem}

\begin{proof}
From Theorem \ref{3t1}, we know that (i) is equivalent to  Theorem \ref{3t1} (iii) with $p=1$ for $K\in C([0,T_c))$. Thus,  it suffices for us to show that each of  (ii) and (iii) is equivalent to Theorem \ref{3t1} (iii) with $p=1$ for $K\in C([0,T_c))$.

(a) \emph{Theorem \ref{3t1} (iii) with $p=1$ for $K\in C([0,T_c))$  implies\ (ii)\ and \ (iii).}
We again prove this assertion for $s=0$.
 By an approximation argument, it suffices to prove for $f\in C^{\infty}(M)$ such that $\inf f>0$
and $f$ is constant outside a compact set.  In this case, given $t\in (0,T_c)$, let
$$q(s)=1+\frac{(q_1-1)\int_0^te^{2\int_0^rK(u)d u}d r}{\int_0^se^{2\int_0^rK(u)d u}d r}\  \ \mbox{and} \ \ \psi(s)=\{P_{0,s}(P_{s,t}f)^{q(s)}\}^{\frac{1}{q(s)}},$$
for all $s\in (0,t]$.
Then
$$\int_0^se^{-2\int_r^sK(u)d u}d r+\frac{q(s)-1}{q'(s)}=0,$$
which together with \eqref{Ito-eq} implies
\begin{align*}
\l(\frac{\psi'\psi^{q-1}q^2}{q'}\r)(s)
=&P_{0,s}(P_{s,t}f)^{q(s)}\log(P_{s,t}f)^{q(s)}-P_{0,s}(P_{s,t}f)^{q(s)}\log P_{0,s}(P_{s,t}f)^{q(s)}\\
&+\frac{q(s)^2(q(s)-1)}{q'(s)}P_{0,s}(P_{s,t}f)^{q(s)-2}|\nabla^sP_{s,t}f|^2_s.
\end{align*}
Due to Theorem \ref{3t1} (iii) with $p=1$ for $K\in C([0,T_c))$, we further have
\begin{align*}
\l(\frac{\psi'\psi^{q-1}q^2}{q'}\r)(s)\leq q(s)^2\l(\int_0^se^{-2\int_u^tK(r)d r}d
u+\frac{q(s)-1}{q'(s)}\r)P_{0,s}(P_{s,t}f)^{q(s)-2}|\nabla^sP_{s,t}f|^2_s=0.
\end{align*}
Therefore, in case (ii) one has $q'(s)<0 $ so that $\psi'(s)\geq 0$,
while in case (iii) one has $q'(s)>0$ so that $\psi'(s)\leq 0$. Hence,
the inequalities in (ii) and (iii) hold.

(b) \emph{ (ii)\ or\ (iii)\ implies\ Theorem \ref{3t1} (iii) with $p=1$ for $K\in C([0,T_c))$.}  We only
prove that (ii) implies (iii), since ``(ii) implies (iii)" can be shown
in a similar way. We also only consider $s=0$. Let $q_1=2$ and $q_2=2(1+\varepsilon)$ for small
$\varepsilon >0$. According to ($\ref{3e11}$), we take $r(\varepsilon)$ such that
\begin{align*}
\frac{1}{1+2\varepsilon}=\frac{\int_0^{r(\varepsilon)}e^{2\int_0^uK(\tau)d \tau}d
u}{\int_0^te^{2\int_0^uK(\tau)d \tau}d
u}=1-\frac{\int_{r(\varepsilon)}^te^{2\int_0^uK(\tau)d \tau}d
u}{\int_0^te^{2\int_0^uK(\tau)d \tau}d u}.
\end{align*}
Then,
$$\frac{2\varepsilon}{1+2\varepsilon}\int_0^te^{2\int_0^uK(\tau)d \tau}d u=\int_{r(\varepsilon)}^te^{2\int_0^uK(\tau)d \tau}d u \sim ({t-r(\varepsilon)})e^{2\int_0^tK(\tau)d \tau},$$
i.e.
$${(t-r(\varepsilon))}\sim 2\varepsilon \int_0^te^{-2\int_u^tK(\tau)d \tau}d u.$$
So, we obtain from
(ii) that
\begin{align*}
0&\geq \lim_{\varepsilon \rightarrow
0}\frac{1}{\varepsilon}\{[P_{0,r(\varepsilon)}(P_{r(\varepsilon),t})^{2(1+\varepsilon)}]^{\frac{1}{(1+\varepsilon)}}-(P_{0,t}f^2)\}\\
&= P_{0,t}f^2\log f^2-(P_{0,t}f^2)\log
P_{0,t}f^2-4\int_0^te^{-2\int_u^tK(r)d r}d u\cdot
P_{0,t}|\nabla^tf|^2_t.
\end{align*}
Therefore, Theorem \ref{3t1} (iii) with $p=1$  holds for $K\in C([0,T_c))$.
\end{proof}
\medskip

\section{Coupling for  $L_t$-diffusion processes and its applications}
\subsection{Coupling processes}
 We aim to construct coupling processes for $L_t$-diffusion processes by parallel translation and mirror reflection   in this subsection. 

Let us introduce some basic notions first.
 Recall that ${\rm Cut}_t(x)$ is the set of the $g_t$-cut-locus of $x$
on $M$. Then, for each $t\in [0,T_c)$, the $g_t$-cut-locus ${\rm Cut}_t$ and the
space time cut-locus ${\rm Cut_{ST}}$ are defined by
\begin{align*}
{\rm Cut}_t&=\{(x,y)\in M\times M\ |\  y\in {\rm Cut}_{t}(x)\};\\
{\rm Cut_{ST}}&=\{(t,x,y)\in [0,T_c)\times M\times M\ |\  (x,y)\in {\rm
Cut}_t\}.
\end{align*}
Set $D(M)=\{(x,x)|x\in M\}$.
For any $(x,y)\notin {\rm Cut}_t$ with $x\neq y$, let $\{J_i^t\}_{i=1}^{d-1}$
be Jacobi fields along the minimal geodesic $\gamma$ from $x$ to $y$
with respect to the metric $g_t$, such that  $\{J_i^t,
\dot{\gamma}: 1\leq i\leq d-1\}$ is an orthonormal basis at $x$ and $y$. Let
\begin{align*}
  I_Z(t,x,y)=&\sum_{i=1}^{d-1}\int_{\gamma}\l(\l<\nabla^{t}_{\dot{\gamma}}J_i^t, \nabla^{t}_{\dot{\gamma}}J_i^t\r>_t-\l<R_t(J_i^t,\dot{\gamma})\dot{\gamma}, J_i^t\r>_t+\frac{1}{2}\partial_tg_t(\dot{\gamma},\dot{\gamma})\r)(\gamma(s))d s\\
  &+Z_t \rho_t(y, \cdot)(x)+Z_t \rho_t(x, \cdot)(y),
\end{align*}
 where $R_t$ is the curvature tensor with respect to the metric $g_t$.
Moreover, let $P_{x,y}^t: T_xM\rightarrow T_yM$ be the $g_t$-parallel
 translation along the geodesic $\gamma$. Define  the $g_t$-mirror reflection by
$$M^t_{x,y}: T_xM\rightarrow T_yM;\ v\mapsto P_{x,y}^tv-2\l<v,\dot{\gamma}\r>_t(x)\dot{\gamma}(y). $$
 It is well known that $P_{x,y}^t$ and $M_{x,y}^t$
are smooth outside ${\rm Cut}_t$ and $D(M)$. For convenience,
we set $P_{x,x}^t$ and $M_{x,x}^t$ be the identity for $x\in
M$. Our main result in this subsection is the following.
\begin{theorem}\label{3t2}
Let $x\neq y$ and $0<T<T_c$ be fixed. Let $U:[0,T]\times M\times
M\rightarrow TM$ be $C^1$-smooth in $({\rm
Cut_{ST}}\cup[0,T]\times D(M))^{c}$ such that $U(t,x_1,x_2)\in T_{x_2}M$ for $(t,x_1,x_2)\in [0,T]\times M\times
M$.
\begin{enumerate}
  \item[$(i)$] There exist two $\mathbb{R}^d$-valued Brownian motions $B_t$ and $\tilde{B}_t$
 on a complete filtered probability space $(\Omega, \{\mathscr{F}_t\}_{t\geq 0},
 \mathbb{P})$ such that
 $${\bf 1}_{\{(X_t, \tilde{X}_t)\notin{\rm Cut}_t\}}d \tilde{B}_t={\bf 1}_{\{(X_t, \tilde{X}_t)\notin{\rm Cut}_t\}}\tilde{u}_t^{-1}P^t_{X_t,\tilde{X}_t}u_td B_t,$$
 where
  $X_t$ with lift $u_t$ and $\tilde{X}_t$ with lift
$\tilde{u}_t$ solve the following equation
\begin{align*}
\begin{cases}
d X_t=\sqrt{2} u_t\circ d B_t+Z_t(X_t)d t, &
X_0=x,\vspace{0.3cm}
\\
d \tilde{X}_t= \sqrt{2}\tilde{u}_t\circ d
\tilde{B}_t
+\l\{Z_t(\tilde{X}_t)+U(t,X_t,\tilde{X}_t){\bf 1}_{\{X_t\neq
\tilde{X}_t\}}\r\}d t, & \tilde{X}_0=y.
\end{cases}
\end{align*}
Moreover,
for any $J\in C([0,T]\times M\times M)$ such that $J\geq I_Z$ on $({\rm Cut}_{\rm ST}\cup[0,T]\times D(M))^c$,
\begin{align*}
d \rho_t(X_t,
\tilde{X}_t)\leq
\bigg\{&J(t,X_t,\tilde{X}_t) +\l<U(t,X_t,\tilde{X}_t),\nabla^{t}\rho_t(X_t,\cdot)(\tilde{X}_t)\r>_t{\bf
1}_{\{X_t\neq \tilde{X}_t\}}\bigg\}d t
\end{align*}
holds up to the coupling time $T_0:=\inf\{t\in [0,T]: X_t=\tilde{X}_t\}$ with convention $\inf \varnothing =T$.
  \item[$(ii)$] The first assertion in $(1)$ holds by replacing $P_{X_t,\tilde{X}_t}^t$ with
  $M^t_{X_t,\tilde{X}_t}$. In this case, for any $J\in C([0,T]\times M\times M)$ such that $J\geq I_Z$ on $({\rm Cut}_{\rm ST}\cup[0,T]\times D(M))^c$,
\begin{align}\label{3e6}
d \rho_t(X_t, \tilde{X}_t)\leq  2\sqrt{2}d
b_t+\bigg\{&J(t,X_t,\tilde{X}_t)+\l<U(t,X_t,\tilde{X}_t),\nabla^{t}\rho_t(X_t,\cdot)(\tilde{X}_t)\r>_t{\bf
1}_{\{X_t\neq \tilde{X}_t\}}\bigg\}d t\end{align} holds
up to the coupling time $T_0$, where $b_t$ is a
one-dimensional Brownian motion.
\end{enumerate}
\end{theorem}
\begin{proof}
We first deal with the reflecting coupling
case for $U=0$. For the parallel coupling case, the proof is similar by replacing $M_{x,y}^t$ with $P_{x,y}^t$ below. The proof is divided into
two parts.

\emph{Part I: Construction of ($X_t, \tilde{X}_t$).} Let $u_t$ and $X_t:={\bf p}u_t$  solve \eqref{SDE-u}.
To get rid of the trouble that $M^t_{x,y}$ does not exist on ${\rm Cut}_t\cup D(M)$, we modify this operator so that it vanishes in a neighborhood of these sets.
To this end,
for any $n\geq 1$ and $\varepsilon\in (0,1)$, let $h_{n,\varepsilon}\in C^{\infty}([0,T]\times M\times M)$ so that
$$0 \leq h_{n,\varepsilon}\leq (1-\varepsilon), \  h_{n,\varepsilon}|_{{ C}_n^c}={1-\varepsilon},\ \mbox{and}\  \ h_{n,\varepsilon}|_{{C}_{2n}}=0,$$
where $${ C}_n:=\{(t,x,y)\in [0,T]\times M\times M: {\bf d}((t,x,y),{\rm Cut}_{\rm ST})\leq 1/n\}$$ and ${\bf d}$ is a distance on
$[0,T]\times M\times M$ such that, for $(s,x_1,y_1),(t,x_2,y_2)\in [0,T]\times M\times M$,
$${\bf d}((s,x_1,y_1),(t,x_2,y_2))=|t-s|+\sup_{r\in [0,T]}\rho_{r}(x_1,x_2)+\sup_{r\in [0,T]}\rho_r(y_1,y_2).$$
Moreover, let $\varphi_n\in C^{\infty}([0,T]\times M\times M)$ such that $0\leq \varphi_n\leq1$,
$\varphi_n|_{D_{2n}}=0$ and $\varphi_n|_{D_n^c}=1$, where
$$D_n:=\l\{(t,x,y)\in [0,T]\times M\times M: \rho_t(x,y)\leq \frac{1}{n}\r\}.$$
 Let $\tilde{u}^{n,\varepsilon}_t$ and
$\tilde{X}_t^{n,\varepsilon}:={\bf
p}\tilde{u}^{n,\varepsilon}_t$ solve the SDE
\begin{equation}\label{SDE2}
\begin{cases}
 d \tilde{u}_t^{n,\varepsilon}=\sqrt{2}(h_{n,\varepsilon}\varphi_n)(t,X_t,\tilde{X}_t^{n,\varepsilon})\dsum_{i=1}^{d}H_{i}^t(\tilde{u}_t^{n,\varepsilon})\circ d
\tilde{B}_t^{i}-\frac{1}{2}\dsum_{\alpha,
\beta}\mathcal{G}_{\alpha,\beta}( t,\tilde{u}_t^{n,\varepsilon} )V_{\alpha \beta}(\tilde{u}_t^{n,\varepsilon})d t\\
\qquad \qquad
+\sqrt{2[1-(h_{n,\varepsilon}\varphi_n)^2(t,X_t,\tilde{X}_t^{n,\varepsilon})]}\dsum_{i=1}^{d}H_{i}^t(\tilde{u}_t^{n,\varepsilon})\circ
d {B'_t}^{i}+H_{Z_t }^t(\tilde{u}_t^{n,\varepsilon})d
t,\medskip\\
\tilde{u}_0^{n,\varepsilon}\in \mathcal{O}_t(M),\ {\bf
p}\tilde{u}_0^{n,\varepsilon}=y,
\end{cases}
\end{equation}
where $B'_t$ is a Brownian motion on $\mathbb{R}^d$ independent of
$B_t$, and on $({\rm Cut}_{\rm ST})^c$,
 $$d
\tilde{B}_t={(\tilde{u}^{n,\varepsilon}_t)}^{-1}M^t_{X_t,
\tilde{X}_t^{n,\varepsilon}}u_td B_t .$$
Since the coefficients involved in (\ref{SDE2}) are at least $C^{1}$, there exists a solution
$\tilde{u}^{n,\varepsilon}_t$ solving the equation \eqref{SDE2}.
Let
\begin{align}\label{operator}
\tilde{L}_{n,\varepsilon}(t):=&\Delta_t (x)+\Delta_t (y)
+Z_t(x)+Z_t(y)
+h_{n,\varepsilon}\varphi_n(t,x,y)\sum_{i,j=1}^{d}\l<M^t_{x,y}V_i,
W_j\r>_tV_iW_j,
\end{align}
where $\{V_i\}$ and $\{W_i\}$ are orthonormal bases at $x$ and $y$
respectively. It is easy to see that
$(X_t,\tilde{X}_t^{n,\varepsilon})$ is generated by
$\tilde{L}_{n,\varepsilon}(t)$ and hence, is a coupling of
$L_t$-diffusion processes as the marginal operators of
$\tilde{L}_{n,\varepsilon}(t)$ coincide with $L_t$.

Now, let $\mathbb{P}^{x,y}_{n,\varepsilon}$ be the distribution
of $(X_t,\tilde{X}_t^{n,\varepsilon})$,
which is a probability measure on the path space  $M^T_x\times M^T_y$,  where
$M_x^{T}:=\{\gamma\in C([0,T], M):\gamma_0=x\}$ is equipped with
the $\sigma$-field $\mathscr{F}_x^T$ induced by all measurable cylindric functions.
%
   Since  $\{\mathbb{P}^{x,y}_{n,\varepsilon}: n\geq
1, 0<\varepsilon<1\}$ is a family of couplings for $\mathbb{P}^x$ and $\mathbb{P}^y$, it is easy to see that $\{\mathbb{P}^{x,y}_{n,\varepsilon}: n\geq
1, 0<\varepsilon<1\}$ is tight. Then, for each $\varepsilon >0$, there exists a probability measure $\mathbb{P}^{x,y}_{\varepsilon}$  and a subsequence $\{n_k\}$   such
 that $\mathbb{P}^{x,y}_{n_k,\varepsilon}$ converges weakly to
 $\mathbb{P}_{\varepsilon}^{x,y}$  as $k\rightarrow\infty$.  Meanwhile, there exists   a subsequence $\{\varepsilon_l\}$ such that $\mathbb{P}^{x,y}_{\varepsilon_l}$
 converges weakly to some $\mathbb{P}^{x,y}$ as $l\rightarrow\infty$.
Let $$\tilde{L}_t(x,y)=L_t(x)+L_t(y)+1_{({\rm Cut}_t\cup D(M))^{c}}(x,y)\sum_{i,j=1}^d\l<M_{x,y}^tV_i,W_j\r>_tV_iW_j.$$
Then it is easy to see that  $\mathbb{P}^{x,y}$ solves the martingale problem for $\tilde{L}_t$
 up to the coupling time, i.e. for any $f\in C_0^{\infty}(M\times M\setminus D(M))$,
 $$f(\xi_t,\eta_t)-\int_0^t\tilde{L}_sf(\xi_s,\eta_s)d s,\ \ 0 \leq t\leq T$$
 is a $\mathbb{P}^{x,y}$-martingale with respect to the natural filtration up to the coupling time.
Here and in the sequel,  $(\xi_{\cdot}, \eta_{\cdot})\in C([0,T]; M\times M)$
is the canonical path. It  is well known  that solutions to martingale problem for $\tilde{L}_s$ can be constructed as solutions to a stochastic differential equation.
Therefore, there exist two independent $d$-dimensional Brownian motions $B_t$ and $B'_t$ on a complete
 probability space $(\Omega, \mathscr{F}, \mathscr{F}_t,\mathbb{P})$, and two processes $X_t$, $\tilde{X}_t$
 such that
 $$
 \begin{cases}
d^{{\rm It\hat{o}}} X_t=\sqrt{2} u_td B_t+Z_t(X_t)d t, &
X_0=x,\vspace{0.3cm}
\\
d^{{\rm It\hat{o}}} \tilde{X}_t= \sqrt{2}{\bf 1}_{{\rm Cut}_t^c}(X_t, \tilde{X}_t)M^t_{X_t,\tilde{X}_t}u_t d{B}_t+ \sqrt{2}{\bf 1}_{{\rm Cut}_t}(X_t, \tilde{X}_t)\tilde{u}_td B'_t
+Z_t(\tilde{X}_t)d t, & \tilde{X}_0=y
\end{cases}
$$
holds up to $T_0:=\inf\{t\in[0,T]:X_t=\tilde{X}_t\}$, where $d^{{\rm It\hat{o}}}$ stands for the It\^{o}
 differential operator.
 By letting $\tilde{X}_t=X_t$ after $T_0$, we complete the proof of the first assertion.
\medskip

\emph{Part II: Proof of (\ref{3e6}).} (a)
  We only consider noncompact $M$.
For the compact case, the proof is simpler by dropping the stopping time
  $\tau$ below.
Since the generator $\tilde{L}_{n,\varepsilon}(t)$ is strictly elliptic and all the coefficients involved are
  at least $C^1$, and hence by Remark \ref{rem-fun},
    $\mathbb{P}^{x,y}((X_t, \tilde{X}_t^{n,\varepsilon})\in \cdot)$ has a density $p^{n,\varepsilon}_t(z,w)$ with respect to $ \mu_t\otimes  \mu_t$. Moreover,
    ${\rm Cut}_t$ is closed  and $ \mu_t\otimes  \mu_t({\rm Cut}_t)=0$,  then it is easy to see that the  set
\begin{align}\label{3e7}
\{t\in[0,T]\ |\ (t, X_t,
\tilde{X}_t^{n,\varepsilon})\in {\rm Cut}_{\rm ST}\}
\end{align}
 has Lebesgue measure zero almost surely.
This implies ${\bf 1}_{{\rm
Cut}_t}(X_t,\tilde{X}_t^{n,\varepsilon})=0$, a.s.
Then,  by  \cite[Theorem 2]{Ku} (the It\^{o} formula for radial process $\rho_t(x,X_t)$, $x\in M$),
we have
\begin{align*}
d \rho_t(X_t,\tilde{X}_t^{n,\varepsilon})=&\sqrt{2}\sum_{i=1}^dh_{n,\varepsilon}\varphi_n(M^t_{X_t,\tilde{X}_t^{n,\varepsilon}}u_te_i)\rho_t(X_t,\cdot)(\tilde{X}_t^{n,\varepsilon})d B_t^i+\sqrt{2}\sum_{i=1}^du_te_i\rho_t(\cdot,\tilde{X}_t^{n,\varepsilon})(X_t)d B_t^i\notag\\
&+\sqrt{2[1-(h_{n,\varepsilon}\varphi)^2](t,X_t,\tilde{X}_t^{n,\varepsilon})}\sum_{i=1}^d\tilde{u}_t^{n,\varepsilon}e_i\rho_t(X_t,\cdot)(\tilde{X}_t^{n,\varepsilon})d {B_t'}^{i}\notag\\
&+(\tilde{L}_{n,\varepsilon}(t)+\partial _t)\rho_t(X_t,\tilde{X}_t^{n,\varepsilon})d t-d l_t^{n,\varepsilon}.
\end{align*}
This implies
\begin{align}\label{3e3}
d \rho_t(X_t,\tilde{X}_t^{n,\varepsilon})\leq&
\sqrt{2(h_{n,\varepsilon}\varphi_n+1)^2+2[1-(h_{n,\varepsilon}\varphi_n)^2]}\
d b^{n,\varepsilon}_t-d l_t^{n,\varepsilon}\notag\\
& +[h_{n,\varepsilon}\varphi_nI_{Z}+(1-h_{n,\varepsilon}\varphi_n)S](t,
X_t,\tilde{X}_t^{n,\varepsilon})d t\notag\\
=&
2\sqrt{(h_{n,\varepsilon}\varphi_n)(t,X_t,\tilde{X}_t^{n,\varepsilon})+1}\
d b^{n,\varepsilon}_t-d l_t^{n,\varepsilon}\notag\\
& +[h_{n,\varepsilon}\varphi_nI_{Z}+(1-h_{n,\varepsilon}\varphi_n)S](t,
X_t,\tilde{X}_t^{n,\varepsilon})d t,
\end{align}
where $b_t^{n,\varepsilon}$ is a one-dimensional Brownian motion,
$l_t^{n,\varepsilon}$ is an increasing process which increases only
when $(X_t,\tilde{X}_t^{n,\varepsilon})\in {\rm Cut}_t$, and for $(t,x,y)\notin {\rm Cut}_{\rm ST}$,
$$S(t,x,y):=L_t\rho_t(\cdot,
y)(x)+L_t\rho_t(x,\cdot)(y)+\partial_t\rho_t(x,y).$$
Let  $\mathbf{B}$ be a fixed bounded smooth open
domain in $M$. Given $N\geq 1$, the Laplacian comparison theorem
implies that there exists a constant $C>0$ such that $S(t,x,y)\leq
C$ for all $(t,x,y)\in ([0,T]\times \overline{\mathbf{B}}\times \overline{\mathbf{B}})\cap ({\rm Cut}_{\rm ST}\cup D_N)^c$.
  Then, it follows
from (\ref{3e3}) that, when $(t,X_t,\tilde{X}_t^{n,\varepsilon})\in ([0,T]\times \mathbf{B}\times \mathbf{B})\cap D_N^c$ and $n\geq N$,
\begin{align}\label{3e8}
d \rho_t(X_t,\tilde{X}_t^{n,\varepsilon})=&
2\sqrt{(h_{n,\varepsilon}\varphi_n)(t,X_t,\tilde{X}_t^{n,\varepsilon})+1}\
d b^{n,\varepsilon}_t-d \tilde{l}_t^{n,\varepsilon}
\nonumber\\
&+[h_{n,\varepsilon}\varphi_nJ+(1-h_{n,\varepsilon}\varphi_n)C](t,
X_t,\tilde{X}_t^{n,\varepsilon})d t,
\end{align}
where $\tilde{l}_t^{n,\varepsilon}$
is a larger increasing process. Now let $f\in C^{2}(\mathbb{R})$
with $f'\geq 0$ and $f'|_{[0,1/N]}=0$. By the It\^{o} formula, we
obtain from (\ref{3e8}) that, for
 the coordinate process $(\xi_t,\eta_t)$ with $\tau:=\inf\{t\geq
0: (\xi_t,\eta_t)\notin \mathbf{B}\times \mathbf{B}\}$  and any $n\geq N$,
\begin{align*}
S_t^{n,\varepsilon}(f):=&f\circ\rho_t(\xi_{t\wedge
\tau},\eta_{t\wedge \tau})\\
&-\int_0^{t\wedge
\tau}\l\{2(h_{n,\varepsilon}+1)f''\circ\rho+\l[h_{n,\varepsilon}J+(1-h_{n,\varepsilon})C\r]f'\circ\rho\r\}(s,\xi_s,\eta_s)d
s
\end{align*}
is a $\mathbb{P}^{x,y}_{n,\varepsilon}$-supermartingale, where $\rho(t,\cdot,\cdot):=\rho_t(\cdot,\cdot)$, $t\in [0,T]$.
 Thus, for
any $0\leq t'<t\leq T$ and $\mathscr{F}_{t'}$-measurable nonnegative $\phi\in
C_b(M_x^{T}\times M_y^{T})$, one has
\begin{align}\label{3e9}
\mathbb{E}^{x,y}_{n,\varepsilon}\phi S_t^{n,\varepsilon}(f)\leq
\mathbb{E}^{x,y}_{n,\varepsilon}\phi S_{t'}^{n,\varepsilon}(f),\ \ n\geq
N,\end{align} where 
$\mathbb{E}^{x,y}_{n,\varepsilon}$ is the expectation with respect
to $\mathbb{P}^{x,y}_{n,\varepsilon}$.

(b)
Since   the strict ellipticity of $\tilde{L}_{n,\varepsilon}(t)$  is uniform in $n$, and  all  coefficients of this operator  are  uniformly bounded in $n$ on  any compact set $ K\subset  M\times M$, the density $p_t^{n,\varepsilon}$, $t\in (0,T]$ satisfies a Harnack inequality uniform in $n$ on $ K$ (see \cite[Theorem 1]{Tr}\footnote{By a localization argument, there exist finite local coordinates covering the compact set $K$. Meanwhile,  the corresponding heat kernel is uniformly bounded in $n$ on any local coordinates by \cite[Theorem 1]{Tr} and the reference measures are equivalent to each other on $K$,  thus the heat kernel is uniformly bounded in $n$ on the compact set $K$.}). That is  for each $t\in(0,T]$, there exists a constant $C>0$ such that $p_t^{n,\varepsilon}\leq C$ on $ K$ for all $n\geq 1$. Let $G$ be an open set containing ${\rm Cut}_t$. Then,
$${\mathbb{P}}^{x,y}_{\varepsilon}((\xi_t,\eta_t)\in G\cap K^{\circ})=\varliminf_{k\rightarrow \infty}{\mathbb{P}_{n_k,\varepsilon}^{x,y}((\xi_t,\eta_t)\in G\cap K^{\circ})}\leq C\mu_t\otimes\mu_t(G\cap K^{\circ}),$$
where $K^\circ$ is the inner set of $K$.
Since ${\rm Cut}_t$ is a  closed set of measure zero with respect to $\mu_t\otimes \mu_t$, by letting $G\rightarrow{\rm Cut}_t$ and then  $K\rightarrow M\times M$, we obtain $$\mathbb{P}^{x,y}_{\varepsilon}((t,\xi_t,\eta_t)\in {\rm Cut}_{\rm ST})=\mathbb{P}^{x,y}_{\varepsilon}((\xi_t,\eta_t)\in {\rm Cut}_{t})=0.$$
Therefore,
we have, for any $\delta>0$, there
exists $m\geq 1$ such that
\begin{align}\label{ineq1}
\int_0^t\mathbb{P}^{x,y}_{\varepsilon}((s,\xi_s,\eta_s)\in
C_m)d s\leq \delta.\end{align}
 Moreover, it is easy to see that $C_m$ is  closed since ${{\bf d}}$ is continuous and ${\rm Cut}_{\rm ST}$ is closed (see \cite{MT}), which implies
$$\varlimsup_{k\rightarrow\infty}\mathbb{P}^{x,y}_{n_k,\varepsilon}((s,\xi_s, \eta_s)\in C_m)\leq \mathbb{P}^{x,y}_{\varepsilon}((s,\xi_s,\eta_s)\in C_m), \ \ 0\leq s\leq T.$$
With this and \eqref{ineq1},  we obtain that
\begin{align}\label{ineq2}\varlimsup_{k\rightarrow\infty}\int_0^t\mathbb{P}^{x,y}_{n_k,\varepsilon}((s,\xi_s,
\eta_s)\in C_m)d s\leq \delta.
\end{align}
Let
$$S_t^{\varepsilon}(f)=f\circ
\rho_t(\xi_{t\wedge \tau}, \eta_{t\wedge \tau})-\int_0^{t}{\bf
1}_{\{s<\tau\}}\big[2(2-\varepsilon)f''\circ
\rho+((1-\varepsilon)J+\varepsilon
C)f'\circ\rho\big](s,\xi_s,\eta_s)d s.$$
Combining  this with
 (\ref{3e9}),
(\ref{ineq1}) and  (\ref{ineq2}), and noting that $h_{n,\varepsilon}=1-\varepsilon$ on $C_m^c$ for $n\geq m$,
we obtain that there exists some constant $c_1>0$, such that
\begin{align}\label{e:S-t}
&\mathbb{E}^{x,y}_{\varepsilon}S_t^{\varepsilon}(f)\phi
\\
=&\mathbb{E}^{x,y}_{\varepsilon}\phi\Big\{f\circ
\rho_t(\xi_{t\wedge \tau}, \eta_{t\wedge \tau})-\int_0^t{\bf
1}_{\{s<\tau\}}\big[2(2-\varepsilon)f''\circ
\rho+((1-\varepsilon)J+\varepsilon
C)f'\circ\rho\big](s,\xi_s,\eta_s)d s\Big\}\nonumber\\
=&\lim_{k\rightarrow
\infty}\mathbb{E}_{n_k,\varepsilon}^{x,y}\phi\Big\{f\circ\rho_t(\xi_{t\wedge
\tau}, \eta_{t\wedge \tau})-\int_0^t{\bf
1}_{\{s<\tau\}}\big[2(2-\varepsilon)f''\circ\rho+((1-\varepsilon)J+\varepsilon
C)f'\circ\rho\big](s,\xi_s,\eta_s)d s\Big\}\nonumber\\
\leq &\varliminf_{k\rightarrow
\infty}\mathbb{E}^{x,y}_{n_k,\varepsilon}S_t^{n_k,\varepsilon}(f)\phi+\delta
c_1\leq \varliminf_{k\rightarrow
\infty}\mathbb{E}^{x,y}_{n_k,\varepsilon}\phi S_{t'}^{n_k,\varepsilon}(f)+\delta c_1\nonumber\\
\leq &\varliminf_{k\rightarrow
\infty}\mathbb{E}^{x,y}_{n_k,\varepsilon}\phi\Big\{f\circ\rho_t(\xi_{t'\wedge\tau},\eta_{t'\wedge
\tau})-\int_0^{t'}{\bf
1}_{\{s<\tau\}}\big[2(2-\varepsilon)f''\circ\rho\nonumber\\
&\qquad \qquad \qquad+((1-\varepsilon)J+\varepsilon
C)f'\circ\rho\big](s,\xi_s,\eta_s)d s\Big\}+2\delta c_1\nonumber\\
=&\mathbb{E}^{x,y}_{\varepsilon}\phi S_{t'}^{\varepsilon}(f)+2\delta
c_1.\nonumber
\end{align}
Letting $\delta\rightarrow 0$, we conclude that
\begin{align}\label{3e10}
\mathbb{E}^{x,y}_{\varepsilon}\phi S^{\varepsilon}_t(f)\leq
\mathbb{E}^{x,y}_{\varepsilon}\phi S_{t'}^{\varepsilon}(f).
\end{align}
Similarly, let
$$S_t(f)=f\circ\rho_{t\wedge \tau}(\xi_{t\wedge \tau},\eta_{t\wedge \tau})-\int_0^{t\wedge \tau}\l[J f'\circ\rho+4f''\circ\rho\r](s,\xi_s,\eta_s)d s.$$
Then $S^{\varepsilon}_t(f)\rightarrow S_t(f)$ uniformly as
$\varepsilon \rightarrow 0$. By
(\ref{3e10}) and then with  a same discussion as in \eqref{e:S-t},  we obtain that
$$
\mathbb{E}^{x,y}\phi S_t(f)\leq \mathbb{E}^{x,y}\phi S_{t'}(f),
$$
 for all $t>t'$ and $\mathscr{F}_{t'}$-measurable nonnegative $\phi\in
C_b(M_x^T\times M_y^T)$. This means that $S_t(f)$ is a $\mathbb{P}^{x,y}$-supermartingale.

(c) Now, let $f\in C^{2}(\mathbb{R})$ with $f'\geq 0$ be fixed. For any
$N\geq 1$, let
$$T_N:=\inf\{t\in [0,T]: \rho_t(\xi_t,\eta_t)\leq 1/N\}.$$
One has $T_N\rightarrow T_0$ as $N\rightarrow \infty$. Let us take
$\tilde{f}\in C^2(\mathbb{R})$ such that $\tilde{f}'\geq 0$,
$\tilde{f}'|_{[0,1/(2N)]}=0$ and $\tilde{f}=f$ on $[1/N,
\infty)$. Let
\begin{equation*}
    d N_t(f):=d f\circ\rho_t(\xi_t,\eta_t)-\l[ J f'\circ\rho+4f''\circ\rho\r](t,\xi_t,\eta_t)d t, \ N_0(f):=f\circ\rho_0(x,y).
\end{equation*}
Then due to the concrete choice of $\tilde{f}$, one has
$N_{t\wedge T_N \wedge \tau}(f)=S_{t\wedge T_N\wedge
\tau}(\tilde{f})$ and hence $N_{t\wedge T_N \wedge\tau}$ is a
 supermartingale with respect to $\mathbb{P}^{x,y}$. Letting $N\rightarrow \infty$,
we conclude that $N_{t\wedge T_0\wedge \tau}(f)$ is also a
$\mathbb{P}^{x,y}$-supermartingale. Now,   choosing explicit $f$  leads to obtain \eqref{3e6}.
The rest part  is similar to the proof \cite[Theorem 2.1.1]{Wbook1}(c)(d), we omit it here.

When $U\neq 0$, by replacing  $L_t(y)$ with $L_t(y)+U(t,x,y)$ and with  a similar argument as above,  we then complete the proof.
\end{proof}

\subsection{Transportation-cost inequalities}
In this subsection, we  apply  coupling methods to get some  transportation-cost inequalities.

We denote
$$W_{p,t}(\mu,\nu)=\l(\inf_{\eta\in \mathscr{C}(\mu,\nu)}\int_{M\times M}\rho_t^p(x,y)d \eta(x,y)\r)^{1/p},$$
the Wasserstein distance associated to $p\geq 1$, where $\mathscr{C}(\mu,\nu)$ is the set of all probability measures on $M\times M$ with marginal $\mu,
\nu\in \mathscr{P}(M)$ and $\mathscr{P}(M)$ is  the space of all the  probability measures on $M$.

Our main task in this subsection is to prove the transportation-cost inequalities on the manifolds carrying geometric flows.
 In fact, these inequalities have already investigated by  constructing horizontal diffusion processes, see \cite{ACT1}.  Now we review them by using coupling methods.  Note that when the metric is independent of $t$, the equivalence of (i) and (ii) is due to \cite{Sturm}.
\begin{theorem}\label{3t3}
Let $p\geq 1$ and  $K\in C([0,T_c))$. Then the following assertions
are equivalent to each other.
\begin{enumerate}
  \item [$(i)$] The curvature condition \eqref{CV}  holds for the function $K$.
   \item [$(ii)$] For any $x,y\in M$ and $0\leq s\leq t< T_c$, $$W_{p,t}(\delta_xP_{s,t},\delta_yP_{s,t})\leq \rho_s(x,y)e^{-\int_s^tK(r)d
   r}.$$
  \item [$(ii')$] For any $\nu_1, \nu_2\in \mathscr{P}(M)$ and $0\leq s\leq t< T_c$,
  $$W_{p,t}(\nu_1P_{s,t},\nu_2P_{s,t})\leq W_{p,s}(\nu_1,\nu_2)e^{-\int_s^tK(r)d r}.$$
\end{enumerate}

\end{theorem}
\begin{proof}
It is easy to see that (ii') and (ii) are equivalent. It suffices for us to show that ``(i)$\Leftrightarrow$ (ii)".
By using coupling process $(X_t,\tilde{X}_t)$
by parallel displacement in Theorem \ref{3t2} with $U=0$, we obtain
from (i) that
\begin{align*}W_{p,t}(\delta_xP_{s,t},\delta_yP_{s,t})&\leq
\l\{\mathbb{E}\l(\rho_t(X_t,\tilde{X}_t)^p\big|(X_s,
\tilde{X}_s)=(x,y)\r)\r\}^{1/p}\leq \rho_s(x,y)e^{-\int_s^t K(u)d u}.\end{align*}
 That is, (i) implies (ii).
On the other hand, if (ii) holds, then letting $\Pi _{x,y}$ be the
optimal coupling for $\delta_x P_{s,t}$ and $\delta_y P_{s,t}$ for
the $L^p$-transportation cost for $f\in C^1(M)$ and $f$ is constant outside a compact set, we have
\begin{align*}
&|\nabla^sP_{s,t}f|_s\leq \lim_{y\rightarrow
x}\frac{\int_{M\times M}|f(x')-f(y')|\Pi_{x,y}(d x',d
y')}{\rho_s(x,y)}\\
\leq &\lim_{y\rightarrow x}\l[\int_{M\times
M}\l(\frac{|f(x')-f(y')|}{\rho_t(x',y')}\r)^{p/(p-1)}\Pi_{x,y}(d
x',d y')\r]^{(p-1)/p}\cdot
\frac{W_{p,t}(\delta_xP_{s,t},
\delta_yP_{s,t})}{\rho_s(x,y)}\\
\leq & e^{-\int_s^tK(u)d
u}\l(P_{s,t}|\nabla^{t}f|_t^{p/(p-1)}\r)^{(p-1)/p}.
\end{align*}
By Theorem \ref{3t1} ``(ii)$\Rightarrow$(i)", this implies (i).
\end{proof}

\begin{remark}\label{coupling-rem}
Actually, coupling methods are powerful tools for investigating similar problems on  other general spaces.
\begin{itemize}
  \item [(i)] Recently,  transportation-cost inequalities on the path space of $L_t$-diffusion space have been investigated in \cite{Ch2} by construcing suitable coupling process.
  \item [(ii)] We would like to indicate that the dimension-free Harnack inequalities estabalished in Theorem \ref{Har}  also can be proved by using coupling methods. Indeed,  the author investigated dimension-free Harnack inequalities for reflecting diffusion semigroups on time-varing manifold when $\partial M\neq \varnothing$, see \cite{Ch3} for details.
\end{itemize}
\end{remark}

\bigskip
\bigskip

\noindent\textbf{Acknowledgements}  \
The author was supported in part by NSFC (Grant No. C10915252), Zhejiang Provincial Natural Science Foundation of China (Grant No. GB16021090058) and the Natural Science Foundation of Zhejiang University of Technology (Grant No. 2014XZ011).

\bibliographystyle{plain}%

\bibliography{Bib-c2015}

\def\cprime{$'$}
\begin{thebibliography}{10}

\bibitem{ACT}
Marc Arnaudon, Kolehe~Abdoulaye Coulibaly, and Anton Thalmaier.
\newblock Brownian motion with respect to a metric depending on time:
  definition, existence and applications to {R}icci flow.
\newblock {\em C. R. Math. Acad. Sci. Paris}, 346(13-14):773--778, 2008.

\bibitem{ACT1}
Marc Arnaudon, Kol{\'e}h{\`e}~Abdoulaye Coulibaly, and Anton Thalmaier.
\newblock Horizontal diffusion in {$C^1$} path space.
\newblock In {\em S\'eminaire de {P}robabilit\'es {XLIII}}, volume 2006 of {\em
  Lecture Notes in Math.}, pages 73--94. Springer, Berlin, 2011.

\bibitem{Bakry}
D.~Bakry.
\newblock On {S}obolev and logarithmic {S}obolev inequalities for {M}arkov
  semigroups.
\newblock In {\em New trends in stochastic analysis ({C}haringworth, 1994)},
  pages 43--75. World Sci. Publ., River Edge, NJ, 1997.

\bibitem{Bakry94}
Dominique Bakry.
\newblock L'hypercontractivit\'e et son utilisation en th\'eorie des
  semigroupes.
\newblock In {\em Lectures on probability theory ({S}aint-{F}lour, 1992)},
  volume 1581 of {\em Lecture Notes in Math.}, pages 1--114. Springer, Berlin,
  1994.

\bibitem{BE}
Dominique Bakry and Michel {\'E}mery.
\newblock Hypercontractivit\'e de semi-groupes de diffusion.
\newblock {\em C. R. Acad. Sci. Paris S\'er. I Math.}, 299(15):775--778, 1984.

\bibitem{BBG}
Fabrice Baudoin, Michel Bonnefont, and Nicola Garofalo.
\newblock A sub-{R}iemannian curvature-dimension inequality, volume doubling
  property and the {P}oincar\'e inequality.
\newblock {\em Math. Ann.}, 358(3-4):833--860, 2014.

\bibitem{Bismut}
Jean-Michel Bismut.
\newblock {\em Large deviations and the {M}alliavin calculus}, volume~45 of
  {\em Progress in Mathematics}.
\newblock Birkh\"auser Boston, Inc., Boston, MA, 1984.

\bibitem{CW94}
Mu-Fa Chen and Feng-Yu Wang.
\newblock Application of coupling method to the first eigenvalue on manifold.
\newblock {\em Progr. Natur. Sci. (English Ed.)}, 5(2):227--229, 1995.

\bibitem{CW97a}
Mu-Fa Chen and Feng-Yu Wang.
\newblock Estimation of spectral gap for elliptic operators.
\newblock {\em Trans. Amer. Math. Soc.}, 349(3):1239--1267, 1997.

\bibitem{CW97b}
Mu-Fa Chen and Feng-Yu Wang.
\newblock General formula for lower bound of the first eigenvalue on
  {R}iemannian manifolds.
\newblock {\em Sci. China Ser. A}, 40(4):384--394, 1997.

\bibitem{CCM}
Xin Chen, Li-Juan Cheng, and Jing Mao.
\newblock A probabilistic method for gradient estimates of some geometric
  flows.
\newblock {\em Stochastic Process. Appl.}, 125(6):2295--2315, 2015.

\bibitem{Ch1}
Li-Juan Cheng.
\newblock An integration by parts formula on path space over manifolds carrying
  geometric flow.
\newblock {\em Sci. China Math.}, 58(7):1511--1522, 2015.

\bibitem{Ch2}
Li-Juan Cheng.
\newblock Transportation-cost inequalities on path spaces over manifolds
  carrying geometric flows.
\newblock {\em To appear in Bull. Sci. Math.}, 2015.

\bibitem{Ch3}
Li-Juan Cheng and Zhang Kun.
\newblock Reflecting diffusion processes on manifolds carrying geometric flow.
\newblock {\em To appear in J. Theoret. Probab.}, 2015.

\bibitem{C11}
Kol{\'e}h{\`e}~A. Coulibaly-Pasquier.
\newblock Brownian motion with respect to time-changing {R}iemannian metrics,
  applications to {R}icci flow.
\newblock {\em Ann. Inst. Henri Poincar\'e Probab. Stat.}, 47(2):515--538,
  2011.

\bibitem{Cr}
M.~Cranston.
\newblock Gradient estimates on manifolds using coupling.
\newblock {\em J. Funct. Anal.}, 99(1):110--124, 1991.

\bibitem{Do83}
Jozef Dodziuk.
\newblock Maximum principle for parabolic inequalities and the heat flow on
  open manifolds.
\newblock {\em Indiana Univ. Math. J.}, 32(5):703--716, 1983.

\bibitem{EL}
K.~D. Elworthy and X.-M. Li.
\newblock Formulae for the derivatives of heat semigroups.
\newblock {\em J. Funct. Anal.}, 125(1):252--286, 1994.

\bibitem{Gu}
Christine~M. Guenther.
\newblock The fundamental solution on manifolds with time-dependent metrics.
\newblock {\em J. Geom. Anal.}, 12(3):425--436, 2002.

\bibitem{HN}
Robert Haslhofer and Aaron Naber.
\newblock Weak solutions for the {R}icci flow {I}.
\newblock {\em J. Eur. Math. Soc.(to appear)}, 2016.

\bibitem{ll62}
A.~M. Il{\cprime}in, A.~S. Kala{\v{s}}nikov, and O.~A. Ole{\u\i}nik.
\newblock Second-order linear equations of parabolic type.
\newblock {\em Uspehi Mat. Nauk}, 17(3 (105)):3--146, 1962.

\bibitem{Ken}
Wilfrid~S. Kendall.
\newblock Nonnegative {R}icci curvature and the {B}rownian coupling property.
\newblock {\em Stochastics}, 19(1-2):111--129, 1986.

\bibitem{Ku3}
Kazumasa Kuwada.
\newblock Convergence of time-inhomogeneous geodesic random walks and its
  application to coupling methods.
\newblock {\em Ann. Probab.}, 40(5):1945--1979, 2012.

\bibitem{KR}
Kazumasa Kuwada and Robert Philipowski.
\newblock Coupling of {B}rownian motions and {P}erelman's {$\mathscr
  L$}-functional.
\newblock {\em J. Funct. Anal.}, 260(9):2742--2766, 2011.

\bibitem{Ku}
Kazumasa Kuwada and Robert Philipowski.
\newblock Non-explosion of diffusion processes on manifolds with time-dependent
  metric.
\newblock {\em Math. Z.}, 268(3-4):979--991, 2011.

\bibitem{LXM}
Xue-Mei Li.
\newblock Strong {$p$}-completeness of stochastic differential equations and
  the existence of smooth flows on noncompact manifolds.
\newblock {\em Probab. Theory Related Fields}, 100(4):485--511, 1994.

\bibitem{LR}
Torgny Lindvall and L.~C.~G. Rogers.
\newblock Coupling of multidimensional diffusions by reflection.
\newblock {\em Ann. Probab.}, 14(3):860--872, 1986.

\bibitem{MT}
Robert~J. McCann and Peter~M. Topping.
\newblock Ricci flow, entropy and optimal transportation.
\newblock {\em Amer. J. Math.}, 132(3):711--730, 2010.

\bibitem{Thalmaier}
Anton Thalmaier.
\newblock On the differentiation of heat semigroups and {P}oisson integrals.
\newblock {\em Stochastics Stochastics Rep.}, 61(3-4):297--321, 1997.

\bibitem{TW}
Anton Thalmaier and Feng-Yu Wang.
\newblock Gradient estimates for harmonic functions on regular domains in
  {R}iemannian manifolds.
\newblock {\em J. Funct. Anal.}, 155(1):109--124, 1998.

\bibitem{Tr}
Neil~S. Trudinger.
\newblock Pointwise estimates and quasilinear parabolic equations.
\newblock {\em Comm. Pure Appl. Math.}, 21:205--226, 1968.

\bibitem{Sturm}
Max-K. von Renesse and Karl-Theodor Sturm.
\newblock Transport inequalities, gradient estimates, entropy, and {R}icci
  curvature.
\newblock {\em Comm. Pure Appl. Math.}, 58(7):923--940, 2005.

\bibitem{W97}
Feng-Yu Wang.
\newblock Logarithmic {S}obolev inequalities on noncompact {R}iemannian
  manifolds.
\newblock {\em Probab. Theory Related Fields}, 109(3):417--424, 1997.

\bibitem{WEq}
Feng-Yu Wang.
\newblock Equivalence of dimension-free {H}arnack inequality and curvature
  condition.
\newblock {\em Integral Equations Operator Theory}, 48(4):547--552, 2004.

\bibitem{Wbook1}
Feng-Yu Wang.
\newblock {\em Functional Inequalities, Markov Semigroups and Spectral Theory}.
\newblock The Science Series of the Contemporary Elite Youth. Science Press,
  Beijing, 2005.

\bibitem{W09}
Feng-Yu Wang.
\newblock Second fundamental form and gradient of {N}eumann semigroups.
\newblock {\em J. Funct. Anal.}, 256(10):3461--3469, 2009.

\bibitem{W10}
Feng-Yu Wang.
\newblock Harnack inequalities on manifolds with boundary and applications.
\newblock {\em J. Math. Pures Appl. (9)}, 94(3):304--321, 2010.

\bibitem{Wbook2}
Feng-Yu Wang.
\newblock {\em Analysis for diffusion processes on {R}iemannian manifolds}.
\newblock Advanced Series on Statistical Science \& Applied Probability, 18.
  World Scientific Publishing Co. Pte. Ltd., Hackensack, NJ, 2014.

\end{thebibliography}

\end{document}